\documentclass[12pt]{article}

\usepackage{amsmath,amsfonts,amssymb}
\usepackage{bbm}
\newtheorem{defn}{Definition}[section]
\newtheorem{theo}[defn]{Theorem}
\newtheorem{lem}[defn]{Lemma}
\newtheorem{prop}[defn]{Proposition}
\newtheorem{cor}[defn]{Corollary}
\newtheorem{rem}[defn]{Remark}
\newtheorem{exam}[defn]{Example}
\newenvironment{proof}{{\bf Proof }}{{\vskip 0.1cm \hfill$\Box$}}

\def\N {{\mathbb N}}

\def\R {{\mathbb R}}
\def\E{{\mathbb E}}
\def\P{{\mathbb P}}
\def\M{{\mathbb M}}
\newcommand{\F}{\mathcal{F}}

\begin{document}
\noindent
{{\Large\bf Existence, uniqueness and ergodic properties for time-homogeneous It\^o-SDEs with locally integrable drifts and 
Sobolev diffusion coefficients}
{\footnote{This research was supported by Basic Science Research Program through the National Research Foundation of Korea(NRF) funded by the Ministry of Education(NRF-2017R1D1A1B03035632).}}\\ \\
\bigskip
\noindent
{\bf Haesung Lee},
{\bf Gerald Trutnau}
}\\

\noindent
{\small{\bf Abstract.}  
Using elliptic and parabolic regularity results in $L^p$-spaces and generalized Dirichlet form theory, we construct for every starting point weak solutions to SDEs in $\mathbb{R}^d$ up to their explosion times including the following conditions. For arbitrary but fixed $p>d$ the diffusion coefficient $A=(a_{ij})_{1\le i,j\le d}$ is locally uniformly strictly elliptic with functions $a_{ij}\in H^{1,p}_{loc}(\mathbb{R}^d)$ and the drift coefficient $\mathbf{G}=(g_1,\dots, g_d)$ consists of functions $g_i\in L^p_{loc}(\mathbb{R}^d)$. The solution originates by construction from a Hunt process with continuous sample paths on the one-point compactification of $\mathbb{R}^d$ and the corresponding SDE is by a known local well-posedness result pathwise unique up to an explosion time.  Just under the given assumptions we show irreducibility and the strong Feller property on  $L^{1}(\mathbb{R}^d,m)+L^{\infty}(\mathbb{R}^d,m)$ of its transition function, and the strong Feller property on $L^{q}(\mathbb{R}^d,m)+L^{\infty}(\mathbb{R}^d,m)$, $q=\frac{dp}{d+p}\in (d/2,p/2)$, of its resolvent, which both include the classical strong Feller property. We present moment inequalities and classical-like non-explosion criteria for the solution which lead to pathwise uniqueness results up to infinity under presumably optimal general non-explosion conditions. We further present explicit conditions for recurrence and ergodicity, including existence as well as uniqueness of invariant probability measures. \\



\section{Introduction}\label{one}
Consider the stochastic differential equation (SDE)\\
\begin{eqnarray}\label{Ito1}
X_t= x_0 + \int_0^t\sigma(X_s)dW_s + \int^{t}_{0} \mathbf{G}(X_s) ds, \ 0\le t< \zeta,\ x_0\in \mathbb{R}^d, 
\end{eqnarray}
where $W=(W^1,\ldots,W^m)$ is a standard $m$-dimensional Brownian motion starting from zero, $A=(a_{ij})_{1\le i,j\le d}=\sigma\sigma^T$, $\sigma=(\sigma_{ij})_{1\le i\le d, 1\le j\le m}$ and $\mathbf{G}=(g_1,\ldots,g_d)$ are measurable and 
$$
\zeta:=\inf\{t\ge 0\, : \, X_t\notin \R^d \}=\lim_{n\to \infty}\inf\{t\ge 0\, : \, X_t\notin B_n \}
$$
is the explosion time (or life time) of $X$, i.e. the time when $X$ has left any open Euclidean ball $B_n$ of radius $n$ about the origin.
By a classical result, if $\sigma, \mathbf{G}$ consist of locally Lipschitz continuous functions and satisfy a linear growth condition, then (\ref{Ito1}) with $\zeta=\infty$ has a pathwise unique solution that is strong, i.e. adapted to the filtration generated by $W$ (\cite[IV. Theorems 2.4 and 3.1]{IW89}). Note that the just mentioned reference and most of those below also treat the time inhomogeneous case but we only discuss results in the time homogeneous case, i.e. results related to (\ref{Ito1}). \\
We call a solution that is pathwise unique and strong up to $\zeta$ ($\zeta$ being possibly finite, cf. \cite[IV. Definition 2.1]{IW89}) strongly unique up to $\zeta$. Strong uniqueness results for (\ref{Ito1}) with $\zeta=\infty$ for only measurable coefficients were given starting from \cite{Zvon}, \cite{Ver79}, \cite{Ver81}. In these works $\sigma$ is non-degenerate and $\sigma, \mathbf{G}$ are bounded. Regarding bounded coefficients  one can also mention the later work
\cite{Ba}. To our knowledge the first strong uniqueness results for unbounded measurable coefficients start with \cite[Theorem 2.1]{GyMa}, while weak existence results appeared to exist earlier (cf. introduction of \cite{GyMa}). In \cite[Theorem 2.1]{GyMa} $\sigma$ may be chosen locally Lipschitz, with $\sigma\sigma^T$ globally uniformly strictly elliptic and $g_i\in L^{2(d+1)}_{loc}(\R^d)$ with the following growth condition to ensure non-explosion (\cite[Assumption 2.1]{GyMa}): there exists a constant $M\ge 0$ and a non-negative function $F \in L^{d+1}(\R^d)$ such that almost everywhere
$$
\|\mathbf{G}\|=\left (\sum_{i=1}^d g_i^2\right )^{1/2} \le M+F.
$$
Note that this growth condition does not allow for linear growth and that it depends only on almost every point, which is natural since integrals such as the one in (\ref{Ito1}) involving $\mathbf{G}$ should not depend on the particular Borel version chosen for $\mathbf{G}$. In \cite{Zh05}, the following result was obtained: if $\sigma$ consists of continuous functions and is globally uniformly non-degenerate, i.e. $A(x)\ge C\cdot \text{Id}$ in the quadratic form sense for some  constant $C>0$ and every $x\in \R^d$ and $g_i,\partial_k \sigma_{ij}\in L^{2(d+1)}_{loc}(\R^d)$ for any $i,j,k$, then (\ref{Ito1}) has a strongly unique solution up to its explosion time. In \cite[Theorem 1.1(i) and (ii)]{Zh05} two non-explosion conditions are given. Both require the global
boundedness of $\sigma$ and then only depend on $\mathbf{G}$. The first one is similar to the one of \cite{GyMa} given above. The second one is as follows: there exist a constant $M\ge 0$,  and vector fields $\mathbf{H}$, $\mathbf{F}_i$, with $\|\mathbf{F}_i \| \in L^{p_i}(\R^d)$, $p_i\ge 2(d+1)$, such that almost everywhere
$$
\mathbf{G}=\sum_{i=1}^k \mathbf{F}_i +\mathbf{H} \ \ \text{with } \ \ \|\mathbf{H}(x)\|\le M\left (1+1_{\{\|x\|>e\}}\|x\|\log\|x\|\right ).
$$
This non-explosion condition allows for linear growth and can cover singularities of $\mathbf{G}$, a phenomenon that can not occur for SDEs with continuous coefficients, since these are of course locally bounded. Prior to \cite{Zh05}, the following  was obtained in \cite{KR}: if $\sigma$ is the identity matrix, so that the local martingale part in (\ref{Ito1}) is just a $d$-dimensional Brownian motion $W=(W^1,\ldots,W^d)$ and $g_i \in L^{p}_{loc}(\R^d)$, $1\le i\le d$, for some $p>d$, with
\begin{eqnarray}\label{con1}
\int^{t}_{0} \|\mathbf{G}(X_s)\|^r ds<\infty\ \  \P_{x_0}\text{-almost surely on } \ \{t<\zeta\},
\end{eqnarray}
where $r=2$ and $\P_{x_0}$ is the distribution on the paths starting form $x_0$, then (\ref{Ito1}) has a strongly unique solution up to its explosion time. Besides a global $L_{q^-}L_p$-condition which does not allow for linear growth a rather special and not really explicit non-explosion condition is presented in \cite{KR}. Its formulation is quite long but roughly one can say it is given by 
assuming that $\mathbf{G}$ is the weak gradient of a function $\psi$ which is a kind of Lyapunov function for (\ref{Ito1}). For the precise statement, we refer to \cite[Assumption 2.1]{KR}. The strong uniqueness result of \cite{KR} was generalized among others in \cite[Theorem 1.3]{Zh11} to the case of non-trivial continuous $d\times d$-dispersion matrix $\sigma$  with corresponding locally uniformly strictly elliptic diffusion matrix and $\sigma_{ij}\in H^{1,p}_{loc}(\R^d)$ where $p>d$ is the same as for $\mathbf{G}$, relaxing condition (\ref{con1}) to the natural one, i.e. $r=1$ 
(see also Remark \ref{sectorial}(i)) but no non-explosion condition related to the local conditions of \cite[Theorem 1.3]{Zh11} is given. Only a global $L_p$-integrability condition in space is given in \cite[Theorem 1.2]{Zh11}, which again does not allow for linear growth. Note that \cite[Theorem 1.3]{Zh11} also holds under the conditions of Remark \ref{Zhanggeneral}(ii) and that we can handle this case but disregard it for the reasons mentioned in Remark \ref{Zhanggeneral}. Conditions for non-explosion, as well as irreducibility and strong Feller properties that we will discuss below, were given in \cite[Theorems 1.2 and 1.7]{ZhXi16}. The conditions are formulated for the general time-dependent case but
seem to be
not optimal when restricted to the time-homogeneous case.
For instance, growth conditions are formulated separately, first for locally bounded drift coefficient, then for locally unbounded drift coefficient $\mathbf{G}$, whereas our results show that this is unnecessary in the time-homogeneous case. Moreover the growth conditions for possibly locally unbounded drift coefficient in \cite[Theorem 1.7]{ZhXi16} require local integrability of order $p>2d+2$, global ellipticity, Lipschitz continuity outside a ball, i.e. outside a neighborhood of the singularity, and the norm of the drift needs to satisfy a linear growth condition, but as we will see below, we do not need any of these conditions. On the other hand,  weak differentiability of the solution as well as corresponding moment inequalities for the weak gradient of the solution are presented in \cite[Theorems 1.2 and 1.7]{ZhXi16}, which we both do not consider.
For more discussion on the results of \cite{ZhXi16} when restricted to the time-homogeneous case see below and Remark \ref{compZhZi} below.\\
The strong uniqueness results of \cite{KR}  were also recovered in \cite{FF} using a different method of proof which allowed to obtain additional insight on the solution. For instance, the $\alpha$-H\"older continuity of the solution  for arbitrary $\alpha\in (0,1)$  and the differentiability in $L^2(\Omega\times[0,T],\R^d)$ (here $\Omega$ is the path space) with respect to the initial condition. For the latter result see \cite{FF2}. \\
Now we will describe our results and compare them with previous works. Let $p>d$ be arbitrary but fixed. For $A=(a_{ij})_{1\le i,j\le d}$ and $\mathbf{G}$ satisfying our basic assumptions, i.e. $a_{ij} \in H^{1,p}_{loc}(\R^d)$, $1\le i,j\le d$, such that $A$ is locally uniformly strictly elliptic (cf. (\ref{use}) below), and $\mathbf{G}=(g_1, \dots, g_d) \in L^p_{loc}(\R^d,\R^d)$, we construct a weak solution  to (\ref{Ito1}) up to $\zeta$ using elliptic and parabolic regularity results and generalized Dirichlet form techniques. This is achieved in Theorem \ref{weakexistence} and Remark \ref{weakexistence2} and $\sigma$ can be chosen as in Theorem \ref{weakexistence}(i) or (ii). The weak solution originates by construction from a Hunt process with continuous sample paths on the one-point compactification of $\mathbb{R}^d$ (see Theorem \ref{existhunt}).
Moreover, the combination of bilinear form and PDE techniques allows us just under the above assumptions to obtain  the $L^{[1,\infty]}(\R^d,m)$-strong Feller property of the transition function of the weak solution, and the $L^{[q,\infty]}(\R^d,m)$-strong Feller property, $q=\frac{dp}{d+p}\in (d/2,p/2)$, of its resolvent, which both include the classical strong Feller property (cf. Theorem \ref{1-3reg}, Proposition \ref{regular2} and Lemma \ref{fini1}). 
Additionally, the Hunt process is irreducible in the probabilistic sense and the semigroup of the underlying generalized Dirichlet form is strictly irreducible (cf. Corollary \ref{irreduci}). Using the facts that we obtained from the construction method, the solution can be shown to be non-explosive, if there exists a constant  $M> 0$ and some $N_0\in \N$, such that 
\begin{eqnarray}\label{conservative2}
-\frac{\langle A(x)x, x \rangle}{ \left \| x \right \|^2 +1}+ \frac12\mathrm{trace}A(x)+ \big \langle \mathbf{G}(x), x \big \rangle \leq M\left ( \left \| x \right \|^2+1\right )\left (  {\rm ln}(\left \| x \right \|^2+1)+1\right )
\end{eqnarray}
for a.e. $x\in \R^d\setminus B_{N_0}$. This is proven in Theorem \ref{nonextheo} using supermartingales.
The conditions allow for linear growth, for locally unbounded drifts and an interplay between diffusion and drift coefficients such that (even outside $B_{N_0}$) superlinear growth of $\mathbf{G}$ is possible if $\langle\mathbf{G}(x),x\rangle$ is non-positive and superlinear growth of $\mathbf{G}$ and $A$ is possible if diffusion and drift coefficients compensate each other. \\
Once we have constructed a weak solution up to its explosion time, we concentrate on non-explosion conditions for it.
Under a non-explosion condition or more generally under the assumption of non-explosion of the weak solution, it is a global weak solution, whose corrseponding SDE is strongly unique up to $\infty$ by \cite{Zh11,KR}. This observation was first employed in \cite{RoShTr} and leads to new non-explosion results for the pathwise unique local solutions in \cite{Zh11,KR}. As application of this observation, we obtain strong uniqueness of \eqref{Ito1} up to $\infty$ just under the additional non-explosion condition \eqref{conservative2} (or more generally under any condition that guarantees the non-explosion of our weak solution, see our main Theorem \ref{unique2}). But we obtain more than this.
Namely, the pathwise unique solution $(X_t)_{t\ge 0}$ in Theorem \ref{unique2} is not only strong but satisfies all previously derived properties. Our strong Feller property results generalize the ones obtained in \cite[Propositions 3.2 and 3.8]{AKR} and \cite[Theorem 2.8]{BGS13} and improve the results related to the time-homogeneous case in \cite{ZhXi16}. There  $\M$ should be non-explosive to obtain merely the classical strong Feller property (cf. also Remark \ref{hkeandcomp}(iii)). Also, the irreducibility (in the probabilistic sense) here is just obtained under the mentioned basic assumptions on $A$ and $\mathbf{G}$, whereas the assumptions to obtain irreducibility in \cite{ZhXi16} appear to be quite involved.
Additionally, our method provides implicitly a candidate for an invariant measure as well as for a stationary distribution and we derive several explicit sufficient conditions for recurrence and ergodicity, including existence and uniqueness of invariant probability measures (see Section \ref{subsecrec}). Moreover, we derive moment inequalities for the solution (see Theorem \ref{supestimate} which complements \cite[Proposition 14]{FF}  and \cite[Lemma 3.2 of Section 2.3, Theorem 4.1 of Section 2.4]{Mao}). All these are advantages over the methods that were previously employed in \cite{GyMa}, \cite{KR}, \cite{Zh11}, \cite{ZhXi16}, \cite{FF}, and we are able to generalize and even improve many of the classical results in the time-homogeneous case for locally bounded coefficients (see \cite{Bha} and the standard reference \cite{Pi}) to the case of locally unbounded coefficients (see for instance Remark \ref{conspoint} and Theorem \ref{recurrencepinsky}). Here we should emphasize that in contrast to previous literature where Krylov's estimate is used as a main technique, the only time where we indirectly use Krylov's  estimate is when we apply the local pathwise uniqueness result of \cite[Theorem 1.3]{Zh11} (see our main Theorem \ref{unique2}). For all other properties that we prove, we use the method that we described above.\\
The paper is organized as follows. In Section \ref{2} we introduce the notations that are used throughout the text. 
In Section \ref{3} we develop the analysis to define rigorously the infinitesimal generator $L$ that a solution to (\ref{Ito1}) should have under our assumptions. We first use a result of \cite{St99}, i.e. that a strongly continuous semigroup of contractions and a generalized Dirichlet form on some $L^2$-space associated to an extension of $L$ as in (\ref{representation}) below, can be constructed. For this construction, one needs some weak divergence free property of the anti-symmetric part of the drift. Theorem \ref{1-3} (from \cite[Theorem 2.4.1]{BKRS}) implies that one can obtain this property with respect to a measure $m=\rho\,dx$, where $\rho$ is some strictly positive continuous function, 
under  our basic assumptions on $A=(a_{ij})_{1\le i,j\le d}$ and $\mathbf{G}$. Typically, the density $\rho$ is not explicit and not a probability density but has the regularity $\rho\in H^{1,p}_{loc}(\R^d)$. In the whole article we just use its existence as a tool and do not need to know its explicit form, except in Section \ref{explform} and parts of Section \ref{subsecrec}, see Theorem \ref{rhorecurrence} there. Subsequently, we use the elliptic regularity result Proposition \ref{p;gdegep2.3} (from \cite[Theorem 5.1]{BGS13}) and our parabolic regularity result Theorem \ref{1-3reg} which we derive from results in \cite{ArSe} to obtain the regularity as stated in Proposition \ref{regular2} and ${(\textbf{H2})^{\prime}}$. Following the basic idea from \cite{AKR}, we may then use the  Dirichlet form method to obtain the existence of a Hunt process $\M$ with transition function $(P_t)_{t>0}$ associated to the mentioned extension of $L$, with continuous sample paths on the one point compactification $\R^d_{\Delta}$ of $\R^d$ with $\Delta$ (see Theorem \ref{existhunt}). To obtain its existence we crucially make use of the existence of such a Hunt process for merely almost every starting point which we obtain from \cite{Tr2, Tr5}. Once $\M$ is constructed, we can use standard methods from \cite{IW89} (see Theorem \ref{weakexistence} and Remark \ref{weakexistence2}) to arrive at the identification of a weak solution to (\ref{Ito1}) up to $\zeta$. In Section \ref{4}, we first develop non-explosion criteria for $\M$. The first such statement is obtained in Theorem \ref{nonextheo} by some probabilistic technique using supermartingales which dates back at least to \cite[10.2]{StrVar} (see also \cite[Chapter 5.3]{d} and \cite[Section 6.7]{Pi}). The statement is basically that there exists a strictly positive $C^2$-function on $\R^d$ with nice growth properties at infinity such that $Mu-Lu\ge 0$ a.e. for some constant $M> 0$. In the case of an analytic proof it seems to go back to \cite{has} (see \cite[Theorem 2.4]{da}).  Using  the strong Feller property, the non-explosion conditions of Theorem \ref{nonextheo} can also be recovered from \cite{St99},  as explained in Remark \ref{consstannat}. In Section \ref{subsecrec} we discuss recurrence and other ergodic properties involving and not involving the density $\rho$. As previously mentioned, $\rho$ is usually not explicit but can be assumed to be explicit (if needed) as explained in Remark \ref{rhoexpl}. Using a Harnack inequality from \cite{ArSe}, we then show that the semigroup of the underlying generalized Dirichlet form is strictly irreducible in Corollary \ref{irreduci}(i). Consequently, we can apply explicit volume growth conditions from \cite{GT2} to obtain recurrence (cf. Theorem \ref{rhorecurrence}). In the general case, when $\rho$ is not explicitly known, we can also derive explicit recurrence criteria. Theorem \ref{recurrencepinsky}, that is applicable just under our basic assumptions on $A=(a_{ij})_{1\le i,j\le d}$ and $\mathbf{G}$, generalizes \cite[Chapter 6, Theorem 1.2]{Pi} which assumes the drift to be locally bounded. Moreover the proof of Theorem \ref{recurrencepinsky} is different from the one of \cite[Chapter 6, Theorem 1.2]{Pi} and uses basic results of \cite{GT2}, as well as strict irreducibility from Corollary \ref{irreduci}(i) and Proposition \ref{rectrans}. In Proposition \ref{indeprho}, we derive again just under our basic assumptions on $A$ and $\mathbf{G}$ an explicit criterion for ergodicity of $\M$, including the existence of a unique invariant probability measure.
Section \ref{5} is devoted to the mentioned application to pathwise uniqueness results. \\
\section{Notations}\label{2}
Throughout, we consider the Euclidean space $\R^d$, $d\ge 2$, equipped with the Euclidean inner product $\langle\cdot,\cdot \rangle$, the Euclidean norm $\|\cdot\|$ and the Borel  $\sigma$-algebra $\mathcal{B}(\R^d)$. We write $|\cdot|$ for the absolute value in $\R$. For $r\in \R$, $r>0$ and $x\in \R^d$, let $B_r(x):=\{y\in \R^d\,:\, \|x-y\|<r\}$ and denote its closure by $\overline{B}_r(x)$ (similarly for a subset $A\subset \R^d$, let $\overline{A}$ denote its closure). If $x=0$, we simply write $B_r$ and $\overline{B}_r$. We call a subset $B\subset \R^d$, for which $B=B_r(x)$ for some  $r>0$ and $x\in \R^d$, a ball. Let  $R_{x}(r)$ denote the open cube in $\R^d$ with edge length $r>0$ and center $x \in \R^d$ and denote its closure by $\overline{R}_{x}(r)$. The minimum of two values $a$ and $b$ is denoted by $a\wedge b:=\min(a,b)$ and the maximum is  denoted by $a\vee b:=\max(a,b)$. For two sets $A,B$, we define $A+B:=\{a+b\,:\, a\in A \text{ and } b\in B\}$.\\
The set of all $\mathcal{B}(\R^d)$-measurable $f : \R^d \rightarrow \R$ which are bounded, or nonnegative are denoted by $\mathcal{B}_b(\R^d)$, $\mathcal{B}^{+}(\R^d)$ respectively.  Let $U \subset \R^d$, be an open set. The usual $L^q$-spaces  $L^q(U, \mu)$, $q \in[1,\infty]$ of Borel measurable or classes of Borel measurable functions (depending on the context) are equipped with $L^{q}$-norm $\| \cdot \|_{L^q(U, \mu)}$ with respect to the  measure $\mu$ on $U$  and $L^{q}_{loc}(\R^d,\mu) := \{ f \,:\, f \cdot 1_{U} \in L^q(\R^d, \mu),\,\forall U \subset \R^d, U \text{ relatively compact open} \}$, where $1_A$ denotes the indicator function of a set $A \subset \R^d$. Likewise $L^q_{loc}(\R^d,\R^d, \mu)$ denotes the set of all locally $q$-fold integrable vector fields, i.e. 
$$
L^q_{loc}(\R^d,\R^d, \mu):=\{ \mathbf{G}=(g_1,\ldots,g_d):\R^d\to\R^d\,:\, g_i\in L^q_{loc}(\R^d, \mu), 1\le i\le d\}.
$$
The Lebesgue measure on $\R^d$ is denoted by $dx$ and we write $L^q(\R^d)$, $L^q_{loc}(\R^d)$, $L^q_{loc}(\R^d,\R^d)$ for $L^q(\R^d, dx)$, $L^q_{loc}(\R^d,dx)$,  $L^q_{loc}(\R^d, \R^d, dx)$ respectively. In order to avoid notational complications, we assume that locally integrable functions are whenever necessary pointwisely given (not for instance equivalence classes) and hence measurable. Moreover, 
whenever a function $f$ possesses a continuous version, we will assume it is given by it. However, if in a situation, it should be necessary or important to distinguish between classes and pointwisely given functions, we will mention it. 
If $\mathcal{A}$ is a set of measurable functions $f : \R^d \to \R$, we define $\mathcal{A}_0 : = \{f \in \mathcal{A} \ : $ supp($f$) : = supp($|f| dx$) is compact in $\R^d \}$ and $\mathcal{A}_b$ : = $\mathcal{A} \cap L^{\infty}(\R^d)$. As usual, we also denote the set of continuous functions on $\R^d$, the set of continuous bounded functions on $\R^d$, the set of compactly supported continuous functions in $\R^d$ by $C(\R^d)$, $C_b(\R^d)$, $C_0(\R^d)$, respectively. Two Borel measurable functions $f$ and $g$ are called $\mu$-versions of each other, if $f=g$ $\mu$-a.e.\\
Let $\nabla f : = ( \partial_{1} f, \dots , \partial_{d} f )$, where $\partial_j f$ is the $j$-th weak partial derivative of $f$ on $\R^d$ and $\partial_{ij} f := \partial_{i}(\partial_{j} f) $, $i,j=1, \dots, d$. The Sobolev space $H^{1,q}(U)$, $q \in [1,\infty]$ is defined to be the set of all functions $f \in L^{q}(U)$ for which $\partial_{j} f \in L^{q}(U)$, $j=1, \dots, d$, and 
$H^{1,q}_{loc}(U) : =  \{ f  \,:\;  f \cdot \varphi \in H^{1,q}(U),\,\forall \varphi \in  C_0^{\infty}(U)\}$. 
Here $C_0^{k}(U)$, $k\in \N \cup\{\infty\}$, denotes the set of all $k$-fold continuously differentiable functions with compact support in $U$. 
Let $V$ be a bounded open set in  $\R^d$ (typically a ball $B$) and $f: \overline{V}\rightarrow \R$ be a function. For $\beta \in (0,1)$ define
\[
\mathrm{h\ddot{o}l}_\beta(f,\overline{V}) := \sup\left\{\frac{|f(x)-f(y)|}{\|x-y\|^\beta}: x,y \in \overline{V}, x\not=y\right\} \in [0,\infty],
\]
and the H\"{o}lder continuous functions of order $\beta\in (0,1)$ on $\overline{V}$ by
\[
C^{0,\beta}(\overline{V}) := \{f \in C(\overline{V}): \mathrm{h\ddot{o}l}_\beta (f,\overline{V}) < \infty \}.
\]
Then $C^{0,\beta}(\overline{V})$ is a Banach space with norm
\[ 
\|f\|_{C^{0,\beta}(\overline{V})} := \sup_{x \in \overline{V}} |f(x)| + \mathrm{h\ddot{o}l}_\beta (f,\overline{V}).
\]
The space of all locally H\"{o}lder continuous functions of order $\beta\in (0,1)$ on $\R^d$ is defined by
$$
C^{0,\beta}_{loc}(\R^d):=\{f\,:\, f\in C^{0,\beta}_{loc}(\overline{B}) \text{ for any ball } B\}.
$$	
Let $Q$ be a bounded open set in $\R^d \times \R$ and $g: \overline{Q} \rightarrow \R$ be a function. For $\delta \in (0,1)$ denote
\[
\mathrm{ph\ddot{o}l}_\delta(g,\overline{Q}) := \sup\left\{\frac{|g(x,t)-g(y,s)|}{\left(\|x-y \|+\sqrt{|t-s|}\right)^{\delta}}: \;(x,t), (y,s) \in \overline{Q}, \;(x,t) \not=(y,s)\right\} \in [0,\infty],
\]
and the parabolic H\"{o}lder continuous functions of order $\delta\in (0,1)$ on $\overline{Q}$ by
\[
C^{\delta; \frac{\delta}{2}}(\overline{Q}) := \{g \in C(\overline{Q}): \mathrm{ph\ddot{o}l}_\delta(g,\overline{Q}) < \infty \}.
\]
Then $C^{\delta; \frac{\delta}{2}}(\overline{Q})$ is a Banach space with norm
\[ 
\|g\|_{C^{\delta; \frac{\delta}{2}}(\overline{Q})} := \sup_{(x,t) \in \overline{Q}} |g(x,t)| + \mathrm{ph\ddot{o}l}_\delta(g,\overline{Q}).
\]
$g$ is called locally parabolic H\"older continuous, if for any bounded and open set $Q$, there exists $\delta=\delta(Q)$, such that $g\in C^{\delta; \frac{\delta}{2}}(\overline{Q})$. Here $\delta$ may be different for different $Q$. In particular, if $t\in \R$ is fixed,  we then say that $g(\cdot, t)$ is locally H\"{o}lder continuous with possibly changing H\"older exponents.\\
Given a function $f:\R^d\to \R$ and a Hunt process with cemetary $\Delta$, we follow the convention that $f$ is extended to $\R^d\cup \{\Delta\}$ by setting $f(\Delta)=0$.
\section{Weak solutions via generalized Dirichlet forms and elliptic and parabolic regularity}\label{3}
Let $\phi\in H^{1,2}_{loc}(\R^d)$ be such that the measure $m:=\rho\,dx$, $\rho:=\phi^2$, 
has full support on $\R^d$. Let $H^{1,2}_0(\R^d,m)$ be the closure of $C^{\infty}_0(\R^d)$ in $L^2(\R^d,m)$ with respect to the norm $(\int_{\R^d}(\|\nabla f\|^2+f^2)\,dm)^{1/2}$ and $H^{1,2}_{loc}(\R^d,m):=\{f\,:\, f\cdot \varphi\in H^{1,2}_0(\R^d,m),\ \forall \varphi \in C^{\infty}_0(\R^d)\}$. 
Let $A=(a_{ij})_{1\le i,j\le d}$ with $a_{ij}\in H^{1,2}_{loc}(\R^d,m)$ be a symmetric matrix of functions
and locally uniformly strictly elliptic, i.e. for every (open) ball $B \subset \R^d$ there exist real numbers $\lambda_B, \Lambda_B>0$, such that 
\begin{eqnarray}\label{use}
\lambda_B \left \| \xi \right \|^2 \leq \big \langle A(x) \xi, \xi \big \rangle \ \leq \Lambda_B \left \| \xi \right \|^2  \text{ for all }\; \xi \in \R^d, \; x\in B.
\end{eqnarray}
Let $\mathbf{G}=(g_1,\ldots,g_d)\in L^2_{loc}(\R^d,\R^d,m)$ be such that with
\begin{eqnarray}\label{operator}
Lf:=\frac{1}{2}\sum_{i,j=1}^{d}a_{ij}\partial_i\partial_j f+\sum_{i=1}^{d}g_i\partial_i f,\ f\in C^{\infty}_0(\R^d),
\end{eqnarray}
it holds
\begin{eqnarray}\label{41}
\int_{\R^d} Lf\,dm=0, \ \ \ \ \forall f \in C^{\infty}_0(\R^d). 
\end{eqnarray}
Then it is shown in \cite[Theorem 1.5]{St99} that there exists a closed extension $(L_1, D(L_1))$ on $L^1(\R^d,m)$ of $(L,C_0^{\infty}(\R^d))$ that generates a sub-Markovian $C_0$-semigroup of contractions $(T_t)_{t> 0}$. Restricting $(T_t)_{t> 0}$ to $L^1(\R^d,m)_b$, it is well-known that $(T_t)_{t> 0}$ can be extended to a sub-Markovian $C_0$-semigroup of contractions on each $L^r(\R^d,m)$, $r\in [1,\infty)$. Denote by $(L_r, D(L_r))$ the corresponding closed generator with graph norm
$$
\|f\|_{D(L_r)}:=\|f\|_{L^r(\R^d,m)}+ \|L_r f\|_{L^r(\R^d,m)},
$$
and by $(G_{\alpha})_{\alpha>0}$ the corresponding resolvent. For $(T_t)_{t>0}$ and $(G_{\alpha})_{\alpha>0}$ we do not explicitly denote in the notation on which $L^r(\R^d,m)$-space they act. We assume that this is clear from the context. Moreover, $(T_t)_{t>0}$ and $(G_{\alpha})_{\alpha>0}$ can be uniquely defined on $L^{\infty}(\R^d,m)$, but are no longer strongly continuous there.\\
Writing
\begin{eqnarray}\label{representation}
Lf &=& \frac12\sum_{i,j=1}^da_{ij}\partial_{i}\partial_jf+ \sum_{i=1}^d\beta^{\rho,A}_i \partial_i f+ \sum_{i=1}^d(g_i-\beta^{\rho,A}_i) \partial_i f
\end{eqnarray}
with
\begin{eqnarray}\label{beta}
\beta^{\rho,A}_i & := & \frac12 \sum_{j=1}^d \left (\partial_j a_{ij}+a_{ij}\frac{\partial_j \rho}{\rho} \right ), \ 1\le i\le d,\  \beta^{\rho,A}:=(\beta^{\rho,A}_1,\ldots,\beta^{\rho,A}_d)
\end{eqnarray}
we observe that (\ref{41}) is equivalent to
\begin{eqnarray}\label{41b}
\int_{\R^d} \langle \mathbf{G}-\beta^{\rho,A},\nabla f\rangle\,dm=0, \ \ \ \ \forall f \in C^{\infty}_0(\R^d),
\end{eqnarray}
hence
\begin{eqnarray*}
\int_{\R^d} \widehat{L}f\,dm=0, \ \ \ \ \forall f \in C^{\infty}_0(\R^d),
\end{eqnarray*}
where
\begin{eqnarray}\label{41c}
\widehat{L}f &=& \frac12\sum_{i,j=1}^da_{ij}\partial_{i}\partial_jf+ \sum_{i=1}^d\beta^{\rho,A}_i \partial_i f- \sum_{i=1}^d(g_i-\beta^{\rho,A}_i) \partial_i f.
\end{eqnarray}
Noting that $\widehat{g}_i:=2\beta^{\rho,A}_i-g_i\in L^{2}_{loc}(\R^d,m)$, we see that $L$ and $\widehat{L}$ have the same structural properties, i.e. they are given as the sum of a symmetric second order elliptic differential operator and a divergence free first order perturbation with same integrability condition with respect to the measure $m$. Therefore all what will be derived below for $L$ will hold analogously for $\widehat{L}$. Denote the operators corresponding to $\widehat{L}$ (again defined through \cite[Theorem 1.5]{St99}) by $(\widehat{L}_r, D(\widehat{L}_r))$ for the co-generator on $L^r(\R^d,m)$, $r\in [1,\infty)$, $(\widehat{T}_t)_{t>0}$ for the co-semigroup, $(\widehat{G}_{\alpha})_{\alpha>0}$ for the co-resolvent. By  \cite[Section 3]{St99}, we obtain a corresponding bilinear form with domain 
$D(L_2) \times L^2(\R^d,m) \cup L^2(\R^d,m) \times D(\widehat{L}_2)$ by
\[ 
{\mathcal{E}}(f,g):= \left\{ \begin{array}{r@{\quad\quad}l}
  -\int_{\R^d} L_2 f \cdot g \,dm & \mbox{ for}\ f\in D(L_2), \ g\in L^2(\R^d,m),  \\ 
            -\int_{\R^d} f\cdot \widehat{L}_2 g \,dm  & \mbox{ for}\ f\in L^2(\R^d,m), \ g\in D(\widehat{L}_2). \end{array} \right .
\] 
$\mathcal{E}$ is called the {\it generalized Dirichlet form associated with} $(L_2,D(L_2))$. Using integration by parts, it is easy to see that
\begin{eqnarray}\label{gdf}
{\cal E}(f,g)&=&\frac{1}{2}\int_{\mathbb{R}^d}\langle A\nabla f,\nabla g\rangle \, dm-\int_{\mathbb{R}^d}\langle \mathbf{G}-\beta^{\rho,A},\nabla f\rangle g\, dm,\ \ \ f,g\in C_0^{\infty}(\mathbb{R}^d).
\end{eqnarray}
The following lemma, see \cite[Remark 1.7(iii)]{St99}, will be used later:
\begin{lem}\label{algebra}
Let $u \in D(L_1)_{b}$. Then $u^2 \in D(L_1)_b$ and
\[
L_1 u^2 = \langle A\nabla u,\nabla u \rangle + 2 u L_1 u.
\]
\end{lem}
We are going to restrict our previous assumptions to the ones of the following theorem. The theorem itself  is an immediate consequence of an important result \cite[Theorem 2.4.1]{BKRS} (see also \cite[Theorem 1]{BRS} for the original result), which itself is derived by using elliptic regularity results from \cite{T73} in an essential way.
\begin{theo}\label{1-3}
Let $p>d$ be arbitrary but fixed. Let $A:=(a_{ij})_{1\leq i,j \leq d}$ be a symmetric $d \times d$ matrix of functions  $a_{ij} \in H^{1,p}_{loc}(\R^d)$ satisfying (\ref{use}). 
Let $\mathbf{G}=(g_1, \dots, g_d) \in L^p_{loc}(\R^d,\R^d)$. Then there exists $\rho \in C^{0,1-d/p}_{loc}(\R^d) \cap H^{1,p}_{loc}(\R^d)$ with $\rho(x)>0$ for all $x\in \R^d$ and such that 
\[
 \int_{\R^d} \langle  \mathbf{G}-\beta^{\rho,A}, \nabla \varphi \rangle \rho\,dx=0, \quad  \forall  \varphi \in C_0^{\infty}(\R^d),
\]
with
$$
\beta^{\rho,A}\in L_{loc}^p (\R^d, \R^d).
$$
In particular, setting  
$$
\mathbf{B}=(b_1,\ldots, b_d):=\mathbf{G}-\beta^{\rho,A},
$$ 
we have obtained a representation of an arbitrary $\mathbf{G}\in L^p_{loc}(\R^d,\R^d)$ as the sum of the logarithmic derivative $\beta^{\rho,A}$ associated to $A$ and $\rho$ and a $\rho dx$-divergence free vector field $\mathbf{B}\in L^p_{loc}(\R^d,\R^d)$, namely
\[
\mathbf{G}=\beta^{\rho,A}+\mathbf{B}.
\]
\end{theo}
\begin{rem}\label{Zhanggeneral}
It is possible and not difficult to generalize Theorem \ref{1-3} (and basically everything that follows below) in two directions. We do not do this here because it only leads to technical and notational complications, which are better to be investigated and overcome elsewhere. But all necessary tools can be found in this work. The two  directions are:
\begin{itemize}
\item[(i)]
Theorem  \ref{1-3} also holds with $\R^d$ replaced by any open set $U\subset\R^d$, $H^{1,p}_{loc}(U)$ defined as in Section \ref{2}, and
$$
L^p_{loc}(U):=\{f:f1_{V}\in L^p(U), \forall V \text{relatively compact open with } \overline{V}\subset U \},
$$ 
$$
C^{0,1-d/p}_{loc}(U):=\{f:f\in C^{0,1-d/p}(\overline{V}), \forall V \text{relatively compact open with } \overline{V}\subset U \},
$$ 
by considering an exhaustion with bounded and open sets  $(V_n)_{n\ge1}$ of $U$, i.e.
\[
V_n \subset \overline{V}_n \subset V_{n+1}\; \text{ for all } n \in \N \; \text{ and } \cup_{n=1}^{\infty} V_n=U.
\]
\item[(ii)]
As in \cite[Theorem 2.4.1]{BKRS}, the regularity conditions on  $a_{ij}$, $g_i$, $1\leq i,j \leq d$, can be generalized to $a_{ij}\in H^{1,p_n}(B_n)$ and $g_i\in L^{p_n}(B_n)$ with $p_n>d$. The only interesting case is when $\lim_{n\to \infty } p_n=d$, which leads to a slight but technical improvement of the conditions of Theorem \ref{1-3}. Note that $(B_n)_{n\ge1}$ here is a special exhaustion with bounded and open sets of $\R^d$ but one can generalize this to an arbitrary exhaustion with bounded and open sets  $(V_n)_{n\ge1}$ of $\R^d$. 
\end{itemize}
\end{rem}
From now on unless otherwise stated, we fix one density $\rho$ as in Theorem \ref{1-3} and hence assume that 
$$
A:=(a_{ij})_{1\leq i,j \leq d}, \ \mathbf{G}=(g_1, \ldots, g_d),\ \beta^{\rho,A}=(\beta^{\rho,A}_1,\ldots,\beta^{\rho,A}_d),\ \mathbf{B}=(b_1, \ldots , b_d),
$$
are as in Theorem \ref{1-3} with 
$$
p>d.
$$
This implies all assumptions prior to Theorem \ref{1-3} and we fix from now on the corresponding generalized Dirichlet form $\mathcal{E}$ associated with $(L_2,D(L_2))$ and all the corresponding objects under the assumptions of Theorem \ref{1-3}. As before, we set 
$$
m:=\rho\,dx.
$$
Note, that due to the properties of $\rho$ in Theorem \ref{1-3}, we have that $L_{loc}^p (\R^d)= L_{loc}^p (\R^d, m)$ as well as $L_{loc}^p (\R^d, \R^d)= L_{loc}^p (\R^d, \R^d, m)$. \\ \\
We will use the following result from \cite[Theorem 5.1]{BGS13}, adapted to our needs.
\begin{prop}\label{p;gdegep2.3}
Let $d \ge 2$ and $\mu$ a locally finite (signed) Borel measure on $\R^d$ that is absolutely continuous with respect to Lebesgue measure on $\R^d$. Let $A= (a_{ij})_{1 \le i,j \le d}$ and $p>d$ be as in Theorem \ref{1-3}. Let $h_i, c, f \in L^p_{loc}(\R^d)$ and assume that 
\[
\int_{\R^d} \Big( \sum_{i,j = 1}^{d} \frac{a_{ij}}{2} \partial_{ij} \varphi + \sum_{i=1}^{d} h_i \partial_i \varphi + c \varphi  \Big) \  d\mu = \int_{\R^d}  \varphi  f \, dx, \quad \forall \varphi \in C_0^{\infty} (\R^d),
\]
where $h_i$, $c$ are locally $\mu$-integrable. Then $\mu$ has a density in $H^{1,p}_{loc}(\R^d)$ that is locally H\"older continuous.
\end{prop}
We further state a result originally due to Morrey (see the wrong statement in the original monograph \cite[Theorem 5.5.5']{Mor} and \cite[Theorem 1.7.4]{BKRS} and Corollaries for its correction).
\begin{prop}\label{t;morrey2.4}
Assume $p > d \ge 2$. Let $B' \subset \R^d$ be a ball, $h=(h_1,\ldots,h_d) : B' \to \R^d$ and $c,e : B' \to \R$ such that
\[
h_i \in L^p (B'), 1\le i\le d, \ \  \text{and} \ \ c,e \in L^q(B') \ \ \text{for} \ \ q:=\frac{dp}{d+p}.
\]
Let $A= (a_{ij})_{1 \le i,j \le d}$ be as in Theorem \ref{1-3}. Assume that $u \in H^{1,p}(B')$ is a solution of 
\[
\int_{B'} \sum_{i=1}^d \Big( \partial_i \varphi \Big( \sum_{j=1}^d \frac{a_{ij}}{2} \partial_j u + h_i u  \Big)  \Big) + \varphi (cu + e) \, dx =0, \quad \forall \varphi \in C_0^{\infty}(B'),
\]
Then for every ball $B$ with $\overline{B} \subset B'$, we obtain the estimate
\[
\| u \|_{H^{1,p}(B)} \le c_0 (\| e \|_{L^q(B')} + \| u \|_{L^1(B')}),
\]
where  $c_0 < \infty$ is some constant independent of $e$ and $u$.
\end{prop}
Now, we will apply the standard arguments from \cite{AKR} whose details have been exposed in a very clear way in \cite{BGS13}.
We will briefly explain (until and including Remark \ref{sectorial}) the line of arguments how Propositions \ref{p;gdegep2.3} and \ref{t;morrey2.4} lead to elliptic regularity results for $(G_{\alpha})_{\alpha>0}$ and $(T_t)_{t>0}$ by using well-known arguments (see for instance \cite{AKR}, \cite{BGS13}, or \cite{RoShTr}). However, as we will see later, we will slightly improve some regularity results compared to the just mentioned papers. First, we choose an arbitrary $g\in C_0^{\infty}(\R^d)$, $\alpha>0$. Applying Proposition \ref{p;gdegep2.3} with
$$
\mu=-\rho\, G_{\alpha}g\,dx, \ h_i=\beta^{\rho,A}_i-b_i, \ 1\le i\le d,\ c=-\alpha, f=g\rho\in L^p_{loc}(\R^d),
$$
we obtain $\rho\, G_{\alpha}g\in H^{1,p}_{loc}(\R^d)$. Then, we apply Proposition \ref{t;morrey2.4} with
$$
u=\rho\, G_{\alpha}g, \ h_i=\sum_{j=1}^d\left ( \frac{\partial _j a_{ij}}{2}- (\beta^{\rho,A}_i-b_i)\right ), \ 1\le i\le d, 
$$
and 
$$
c=\alpha, \ e=\rho g\in L^q(B'),
$$
where 
\begin{eqnarray}\label{q}
q:=\frac{dp}{d+p}\in (d/2,p/2).
\end{eqnarray}
By the properties of $\rho$, we obtain 
\begin{eqnarray*}
\| \rho\, G_{\alpha}g \|_{H^{1,p}(B)} \le c_0 \left (\| g \|_{L^q(B',m)} + \| G_{\alpha}g \|_{L^1(B',m)}\right ),
\end{eqnarray*}
where $c_0$ is possibly different form the constant in Proposition \ref{t;morrey2.4}, but also doesn't depend on $g$. The last inequality is easily seen to extend to $g\in L^r(\R^d,m)$, $r\in [q,\infty]$, using the contraction properties of $(G_{\alpha})_{\alpha >0}$.
From that we then get that for any $r\in[q,\infty]$, $\alpha>0$ 
\begin{eqnarray}\label{regres}
\| \rho\, G_{\alpha}g \|_{H^{1,p}(B)} \le c_{0} \left (\| g \|_{L^r(B',m)} + \| G_{\alpha}g \|_{L^1(B',m)}\right ), \ \ \forall g\in L^r(\R^d,m),
\end{eqnarray}
where $c_{0}$ is a constant that may be different for different $\alpha$ and $r$, but  doesn't depend on $g$. Using the contraction properties of $(G_{\alpha})_{\alpha >0}$, (\ref{regres}) immediately implies
\begin{eqnarray}\label{regres2}
\| \rho\, G_{\alpha}g \|_{H^{1,p}(B)} \le c_{0} \| g \|_{L^r(\R^d,m)}, \ \ \forall g\in L^r(\R^d,m),
\end{eqnarray}
where $c_0$ in (\ref{regres}) may be different from $c_0$ in (\ref{regres2}) but has the same properties.\\
Writing $T_0:=id$ and
$$
T_t f=G_1(1-L_r)T_t f, \ f\in D(L_r), \ r\in[q,\infty), \ t\ge 0, 
$$
we can see by (\ref{regres}) that for any $r\in[q,\infty)$, $t\ge0$ 
\begin{eqnarray}\label{regsem}
\| \rho\, T_t f \|_{H^{1,p}(B)} \le c_{0} \| T_t f \|_{D(L_r)}, \ \ \forall f\in D(L_r),
\end{eqnarray}
where $c_{0}$ is a constant that may be different for different $r$, but  doesn't depend on $f$.\\
By Morrey's inequality applied to an arbitrary ball $B$, there exists a constant $c>0$ independent of $f$ such that
\[
\|\widetilde{f}\|_{C^{0,\beta}(\overline{B})} \le c \|f\|_{H^{1,p}(B)},\qquad \forall f \in H^{1,p}(B),
\]
where $\widetilde{f}$ on the left hand side is the unique continuous $dx$-version of $f \in H^{1,p}(B)$ and 
\begin{eqnarray}\label{beta2}
\beta := 1 - d/p. 
\end{eqnarray}
In our situation $\rho \in C^{0,\beta}(\overline{B})$ for any ball $B \subset \R^d$ and since $\inf_{x \in \overline{B}} \rho(x) > 0$, we obtain that $\frac{1}{\rho} \in C^{0,\beta}(\overline{B})$. Now
for $f,g \in C^{0,\beta}(\overline{B})$ it holds $f \cdot g \in C^{0,\beta}(\overline{B})$ and
\begin{equation}\label{product}
	\|f \cdot g\|_{C^{0,\beta}(\overline{B})} \le \|f\|_{C^{0,\beta}(\overline{B})} \|g\|_{C^{0,\beta}(\overline{B})}.
	\end{equation}
For any ball $B$, $t \ge 0, \alpha >0, g \in L^r(\R^d,m), r\in [q,\infty], f\in D(L_r),  r \in [q,\infty)$
	\[\|\rho G_\alpha g\|_{H^{1,p}(B)}, \ \|\rho T_t f\|_{H^{1,p}(B)}\]
	are bounded and so by Morrey's inequality applied to each ball $B$ and (\ref{product}) there exist unique locally H\"{o}lder continuous $m$-versions $R_\alpha g, P_t f$ of $G_\alpha g, T_t f$, where we set 
$$
P_0:=id, 
$$
with
$$
\|R_\alpha g\|_{C^{0,\beta}(\overline{B})} \le  \|\rho^{-1}\|_{C^{0,\beta}(\overline{B})} \|\rho R_\alpha g\|_{C^{0,\beta}(\overline{B})}\le \|\rho^{-1}\|_{C^{0,\beta}(\overline{B})} c\, \| \rho G_\alpha g\|_{H^{1,p}(B)}
$$
and 
$$
\|P_t f\|_{C^{0,\beta}(\overline{B})}  \le \|\rho^{-1}\|_{C^{0,\beta}(\overline{B})} c\, \| \rho T_t f\|_{H^{1,p}(B)}.
$$
Applying  (\ref{regres}), (\ref{regres2}), (\ref{regsem}) to the last two inequalities, we  get for any  $t\ge 0$, $\alpha>0$, $g\in L^r(\R^d,m)$,  $r\in[q,\infty]$, $f\in D(L_r)$,  $r\in[q,\infty)$, and any ball $B'$ with $\overline{B}\subset B'$
\begin{eqnarray}\label{gal}
\|R_\alpha g\|_{C^{0,\beta}(\overline{B})} & \le & c_{0}  \left (\| g \|_{L^r(B',m)} + \| G_{\alpha}g \|_{L^1(B',m)}\right ),
\end{eqnarray}
\begin{eqnarray}\label{gal2}
\|R_\alpha g\|_{C^{0,\beta}(\overline{B})} & \le & c_{0}  \| g \|_{L^r(\R^d,m)},
\end{eqnarray}
\begin{eqnarray}\label{tt}
\|P_t f\|_{C^{0,\beta}(\overline{B})} & \le & c_0 \|T_t f\|_{D(L_r)},
\end{eqnarray}
where $c_{0}$ is a constant that may be different for different $r$ (and different in each inequality \eqref{gal},  \eqref{gal2}, and \eqref{tt}), but  doesn't depend on $f$, nor on $g$.   
We summarize consequences of the derived estimates in the following proposition.
\begin{prop}\label{regular1}
Let $t\ge 0$, $\alpha>0$ be arbitrary and $q, \beta$ be defined as in (\ref{q}), (\ref{beta2}). 
Then under the conditions of Theorem \ref{1-3}, it holds:
\begin{itemize}
\item[(i)] $G_{\alpha}g$ has a locally H\"older continuous $m$-version 
$$
R_{\alpha}g\in 
C^{0,\beta}_{loc}(\R^d), \ \  \forall g\in \bigcup_{r\in [q,\infty]} L^r(\R^d,m).
$$
\item[(ii)] $T_t f$ has a locally H\"older continuous $m$-version 
$$
P_t f\in 
C^{0,\beta}_{loc}(\R^d),\ \  \forall f\in \bigcup_{r\in [q,\infty)} D(L_r).
$$
\item[(iii)] For any $f\in \bigcup_{r\in [q,\infty)} D(L_r)$ the map
$$
(x,t)\mapsto P_t f(x)
$$
is continuous on $ \R^d\times [0,\infty)$.
\end{itemize}
\end{prop}
\begin{proof}
(i) and (ii) are direct consequences of (\ref{gal}), (\ref{gal2}), (\ref{tt}). In order to show (iii), let $f\in D(L_r)$ for some $r\ge q$ and $\left ((x_n,t_n)\right )_{n\ge 1}$ be a sequence in $\R^d\times [0,\infty)$ that converges to $(x_0,t_0)\in \R^d\times [0,\infty)$. Then there exists a ball $B$ such that $x_n\in \overline{B}$ for all $n\ge 0$. By (\ref{tt}) applied with $t=0$ to $P_{t_n}f-P_{t_0} f\in D(L_r)$, noting that $L_r (P_{t_n}f-P_{t_0} f)=P_{t_n}L_rf -P_{t_0}L_r f $ and using the continuity for each $g\in L^r(\R^d,m)$ of $t\mapsto P_t g$ on $[0,\infty)$, we obtain that $P_{t_n}f\to P_{t_0}f$ in $C^{0,\beta}(\overline{B})$.
Then it is clear from (ii) that
\begin{eqnarray*}
\left | P_{t_n} f(x_n)-P_{t_0} f(x_0) \right | \le \left | P_{t_n} f(x_n)-P_{t_0} f(x_n) \right |+\left | P_{t_0} f(x_n)-P_{t_0} f(x_0) \right |
\end{eqnarray*}
converges to zero as $n\to \infty$.
\end{proof}
\begin{rem}\label{sectorial}
\begin{itemize}
\item[(i)] In comparison to \cite{AKR}, \cite{BGS13}, \cite{RoShTr}, we obtained in Proposition \ref{regular1}(i) that $(G_{\alpha})_{\alpha>0}$ is $L^r(\R^d,m)$-strong Feller for any $r\in [q,\infty]$, which is an improvement to the mentioned papers since there it is  only obtained for $r\in [p,\infty]$. This plays a role, since it will imply (\ref{con1}) for $r=2$ . Indeed, we will see  later in Lemma \ref{nest}(ii) that $\int_0^t |f|^2(X_s)ds$ is finite in the sense of (\ref{con1}), whenever $f\in L^{2q}_{loc}(\R^d)$. But  $2q\in (d,p)$, hence $L^{p}_{loc}(\R^d)\subset L^{2q}_{loc}(\R^d)$. 
\item[(ii)] We can use Proposition \ref{regular1}(i) to get a resolvent kernel and a resolvent kernel density for any $x\in \R^d$. 
Indeed, for any $\alpha>0$, $x\in \R^d$, Proposition \ref{regular1}(i) implies that
\begin{eqnarray}\label{resker}
R_{\alpha}(x,A):=\lim_{l\to \infty}R_{\alpha}(1_{B_l\cap A})(x), \ A\in  \mathcal{B}(\R^d)
\end{eqnarray}
defines a  finite measure $R_{\alpha}(x,dy)$ on $(\R^d, \mathcal{B}(\R^d))$  (such that $\alpha R_{\alpha}(x,dy)$ is a  sub-probability measure) that is absolutely continuous with respect to $m$. The Radon-Nikodym derivative
\begin{eqnarray}\label{reskernext}
r_\alpha (x, \cdot):=\frac{R_{\alpha}(x,dy)}{m(dy)}
\end{eqnarray} 
then defines the desired resolvent kernel density.
\item[(iii)] If the $L^2(\R^d,m)$-semigroup $(T_t)_{t>0}$ is analytic (for instance, if the bilinear form in (\ref{gdf}) satisfies a sector condition) then by Stein interpolation $(T_t)_{t>0}$ is also analytic on $L^r(\R^d,m)$ for any $r \in (2,\infty)$ (cf. \cite[Remark 2.5]{RoShTr}). Hence by \cite[Ch. 2, Theorem 5.2(d)]{pazy}, we have for any $r \in [2,\infty)$, $f \in L^r(\R^d,m)$
\[
T_t f \in D(L_r),\ \ \ \text{and }\ \ \  \|L_r T_t f \|_{L^r(\R^d,m)} \le \frac{\text{const.}}{t}\|f\|_{L^r(\R^d,m)}.
\]
Therefore, (\ref{tt}) can be improved and extended as follows: for any $r\in [q\vee 2,\infty)$, $t>0$, $f\in L^r(\R^d,m)$ and any ball $B$
\begin{eqnarray}\label{tt2}
\|P_t f\|_{C^{0,\beta}(\overline{B})} & \le & c_0 \left(1 + \frac{\text{const.}}{t}\right) \|f\|_{L^r(\R^d,m)}.
\end{eqnarray}
We can then use (\ref{tt2}) to get a heat kernel and a heat kernel density for any $x\in \R^d$. 
Indeed, for any $t>0$, $x\in \R^d$, (\ref{tt2}) implies that
\begin{eqnarray}\label{heatker}
P_{t}(x,A):=\lim_{l\to \infty}P_{t}(1_{B_l\cap A})(x), \ A\in  \mathcal{B}(\R^d)
\end{eqnarray}
defines a sub-probability measure $P_{t}(x,dy)$ on $(\R^d, \mathcal{B}(\R^d))$ that is absolutely continuous with respect to $m$. The Radon-Nikodym derivative
\begin{eqnarray}\label{rnd}
p_t (x, \cdot):=\frac{P_{t}(x,dy)}{m(dy)}
\end{eqnarray}
then defines the desired heat kernel density. However, in general $(T_t)_{t>0}$ is not analytic and therefore we cannot impose analyticity.  Moreover, it is in general difficult to check analyticity, in particular the sector condition of the corresponding bilinear form (see for instance \cite[Section 5]{RoShTr}). 
\end{itemize}
\end{rem}
Unfortunately, by what is explained in Remark \ref{sectorial}(iii) the semigroup estimate (\ref{tt}) which leads to Proposition \ref{regular1}(ii) seems just not good enough to obtain a pointwise heat kernel from which one could then try to build a transition function of a nice Markov process. We will proceed by deriving more regularity in the following Theorem \ref{1-3reg}. 
\begin{theo}\label{1-3reg}
Let $A:=(a_{ij})_{1\leq i,j \leq d}, \ \mathbf{G},\ \rho, \ \beta^{\rho,A}$, and $\mathbf{B}$ be as in Theorem \ref{1-3}. For each $s \in [1, \infty]$, consider the $L^s (\R^d, m)$-semigroup $(T_t)_{t>0}$. Then for any $f \in L^s(\R^d, m)$ and $t>0$, $T_t f$ has a continuous $m$-version $P_t f$ on $\R^d$. More precisely, $P_{\cdot}f(\cdot)$ is locally parabolic H\"{o}lder continuous on $\R^d \times (0, \infty)$ and for any bounded open sets $U$, $V$ in $\R^d$ with $\overline{U} \subset V$ and $0<\tau_3<\tau_1<\tau_2<\tau_4$, i.e. $[\tau_1, \tau_2] \subset (\tau_3, \tau_4)$, we have for some $\gamma\in (0,1)$ the following estimate for all $f \in \cup_{s\in[1,\infty]} L^s(\R^d, m)$ with $f \geq 0$,
\begin{equation} \label{thm main est}
\|P_{\cdot} f(\cdot)\|_{C^{\gamma; \frac{\gamma}{2}}(\overline{U} \times [\tau_1, \tau_2])} \leq  C_6 \| P_{\cdot} f(\cdot) \|_{L^1( V \times (\tau_3, \tau_4), m\otimes dt) },
\end{equation}
where $C_6, \gamma$ are constants that depend on $\overline{U} \times [\tau_1, \tau_2], V \times (\tau_3, \tau_4)$, but are independent of $f$. 
\end{theo}
\begin{proof}
First assume $f \in C_{0}^{\infty}(\R^d)$, $f \geq 0$ and set $u(x,t):=\rho(x) P_t f(x)$. Then $f\in D(L_p)$ and by Proposition \ref{regular1}(iii) $P_t f(x)$ is jointly continuous on $\R^d\times[0,\infty)$. Therefore the same is true for $u(x,t)$. Let $\widehat{L}$ be as in (\ref{41c}) and $T>0$ be arbitrary. Then exactly as in \cite[(4.7)]{BKR2} (note that there the underlying measure $m=\mu$ is a probability measure but it doesn't matter), we get for any $\varphi\in C^{\infty}_0(\R^d\times (0,T))$
\begin{eqnarray}\label{**}
0=-\int _0^T\int_{\R^d}  \left ( \partial _t \varphi +\widehat{L}\varphi  \right ) u\,  dxdt. 
\end{eqnarray}
Note that $u \in H^{1,2}(O\times (0,T))$ for any bounded and open set $O\subset \R^d$. We can hence use integration by parts in the right hand term of (\ref{**}) and see that 
$$
0=\int_0^T\int_{\R^d} \left ( \frac{1}{2}\langle A \nabla u, \nabla \varphi \rangle + u \langle \beta, \nabla \varphi \rangle -u\partial_t\varphi \right ) dxdt,
$$
where $\beta:=\frac{1}{2}\nabla A+\mathbf{G}-2\beta^{\rho, A}\in L^{p}_{loc}(\R^d,\R^d)$, $(\nabla A)_i:=\sum_{j=1}^d \partial_j a_{ij}, 1\le i \le d$. \\
Let $\tau_2^*:=\frac{\tau_2+\tau_4}{2}$ and take $r>0$ so that 
$$
r<\min\left (\frac{1}{9}\sqrt{\frac{\tau_4-\tau_2}{14}},\frac{1}{9}\sqrt{\frac{\tau_1}{2}}\right )\quad \text{and }\quad R_{\bar{x}}(9r)\subset V, \ \forall \bar{x}\in \overline{U}.
$$ 
Then for all $(\bar{x}, \bar{t}) \in \overline{U} \times [\tau_1, \tau_2^*]$, we have $\bar{t}-2(9r)^2>0$ and
\[
R_{\bar{x}}(9r) \times (\bar{t}+6(9r)^2, \bar{t}+7(9r)^2)) \subset V \times (\tau_3, \tau_4).
\]
Using \cite[Theorem 4]{ArSe}, for any $(x,t), (y,s) \in R_{\bar{x}}(r) \times (\bar{t}-r^2, \bar{t})$ we have
\[
|u(x,t)-u(y,s)| \leq C_1 r^{-\gamma} \left(\|x-y\|+\sqrt{|t-s|} \right)^{\gamma} \sup_{R_{\bar{x}}(3r) \times (\bar{t}-(3r)^2, \bar{t})} u,
\]
where $C_1$ and $\gamma \leq 1-\frac{d}{p}$ are constants  independent of $f$, $r$ and $(\bar{x}, \bar{t})$.
Thus $u \in C^{\gamma; \frac{\gamma}{2}}\big(\bar{R}_{r}(\bar{x}) \times [\bar{t}-r^2, \bar{t}] \big)$ and
\begin{align*}
\|u\|_{C^{\gamma; \frac{\gamma}{2}}\big(\bar{R}_{r}(\bar{x}) \times [\bar{t}-r^2, \bar{t}] \big)} &\leq (1+C_1 r^{-\gamma} )\sup_{R_{\bar{x}}(3r) \times (\bar{t}-(3r)^2, \bar{t})} u.
\end{align*}
Using the compactness of $\overline{U} \times [\tau_1, \tau_2]$, there exist $(x_i, t_i) \in  \overline{U} \times [\tau_1, \tau^{*}_2]$, $i=1, \dots, N$, such that
\[
\overline{U} \times [\tau_1, \tau_2] \subset \bigcup_{i=1}^{N} R_{x_i}(r) \times (t_i-r^2, t_i)=:Q.
\]
Take a smooth partition of unity $(\phi_i)_{i=1,\dots,N}$ subordinate to $\left (R_{x_i}(r) \times (t_i-r^2, t_i) \right )_{i=1,\dots,N}$. For each $1\leq i \leq N$,\; $\phi_i u \in C^{\gamma; \frac{\gamma}{2}}(\overline{Q})$, so that $u=\sum_{i=1}^{N} \phi_i u$ in $\overline{U} \times [\tau_1, \tau_2]$ implies $u \in C^{\gamma; \frac{\gamma}{2}}(\overline{U}\times [\tau_1, \tau_2])$. Furthermore, we have
\begin{align} \nonumber
\|u\|_{C^{\gamma; \frac{\gamma}{2}}(\overline{U} \times [\tau_1, \tau_2])} &\leq \sum_{i=1}^{N} \left \|\phi_i u  \right \|_{C^{\gamma; \frac{\gamma}{2}}(\overline{U} \times [\tau_1, \tau_2])} \leq \sum_{i=1}^{N} \left \|\phi_i u  \right \|_{C^{\gamma; \frac{\gamma}{2}}(\overline{Q})} \\  \nonumber
&\leq \sum_{i=1}^{N} \|\phi_i \|_{C^{\gamma; \frac{\gamma}{2}}(\overline{Q})} \|u\|_{C^{\gamma; \frac{\gamma}{2}}\big(\bar{R}_{r}(x_i) \times [t_i-r^2, t_i] \big)} \\ \label{1stes}
&\leq \underbrace{\left( \sum_{i=1}^{N} \|\phi_i \|_{C^{\gamma; \frac{\gamma}{2}}(\overline{Q})}\cdot (1+C_1r^{-\gamma} )\right) }_{=:C_2} \cdot \max_{1\leq i \leq N}\left( \sup_{R_{x_i(3r)} \times (t_i-(3r)^2, t_i)} u \right).
\end{align}
Then, by \cite[Theorems 2 and 3]{ArSe}, for each $1 \leq i \leq N$
\begin{align}  \nonumber
\sup_{R_{x_i}(3r) \times (t_i-(3r)^2, t_i)} u &\leq C_3 \|u \|_{L^2(R_{x_i(9r)} \times (t_i-(9r)^2, t_i))} \\  \nonumber
& \leq C_3 (18r)^{\frac{d}{2}} \cdot (9r) \sup_{R_{x_i}(9r) \times (t_i-(9r)^2, t_i)} u \\  \nonumber
& \leq C_3 (18r)^{\frac{d}{2}} \cdot (9r) \cdot C_4 \inf_{R_{x_i}(9r) \times (t_i+6(9r)^2, t_i+7(9r)^2)} u \\  \nonumber
& \leq C_3 C_4  (18r)^{-\frac{d}{2}} \cdot (9r)^{-1}  \| u \|_{L^1(R_{x_i}(9r) \times (t_i+6(9r)^2, t_i+7(9r)^2))} \\  \label{2stes}
&\leq  \underbrace{C_3 C_4  (18r)^{-\frac{d}{2}} (9r)^{-1}}_{=:C_5}  \| u \|_{L^1( V \times (\tau_3, \tau_4)) },
\end{align}
where $C_3$ and $C_4$ are constants which are independent of $f$ and $x_i$.
Combining \eqref{1stes}, \eqref{2stes} we have for $s \in [1, \infty)$
\begin{eqnarray}\label{rest} 
\|P_{\cdot} f(\cdot)\|_{C^{\gamma; \frac{\gamma}{2}}(\overline{U} \times [\tau_1, \tau_2])} & \leq& \| \rho^{-1} \|_{C^{\gamma}(\overline{U} \times [\tau_1, \tau_2])}   \|\rho(\cdot) P_{\cdot} f(\cdot)\|_{C^{\gamma; \frac{\gamma}{2}}(\overline{U} \times [\tau_1, \tau_2])}  \nonumber \\
&\leq&  \underbrace{\| \rho^{-1} \|_{C^{\gamma}(\overline{U} \times [\tau_1, \tau_2])}  C_2 C_5 }_{=:C_6} \| P_{\cdot} f(\cdot) \|_{L^1( V \times (\tau_3, \tau_4), m\otimes dt) }  \nonumber\\ 
& \leq&  C_6  (\tau_4-\tau_3) \|\rho \|_{L^1(V)}^{\frac{s-1}{s}}  \|f\|_{L^s(\R^d, m) }. 
 \end{eqnarray}
For  $f\in L^s(\R^d, m)$, $f\ge 0$, define 
\begin{eqnarray}\label{rot}
P_{\cdot} f(\cdot):=\lim_{n\to \infty}P_{\cdot} f_n(\cdot) \ \text{ in }\ C^{\gamma; \frac{\gamma}{2}}(\overline{U} \times [\tau_1, \tau_2]),
\end{eqnarray}
where $(f_n)_{n\ge 1}\subset C_0^{\infty}(\R^d)$ is any sequence of positive functions converging to $f$ in $ L^s(\R^d, m)$. Then $P_{\cdot} f(\cdot)$ is well-defined, i.e. independent of the choice of $(f_n)_{n\ge 1}$, and \eqref{rest} (including all intermediate inequalities) extends to $f\in L^s(\R^d, m)$, $f\ge 0$. In particular, \eqref{thm main est} holds for $f\in L^s(\R^d, m)$, $f\ge 0$, $s\in[1,\infty)$.\\
Moreover, given $f \in L^s(\R^d, m)$, $f\ge 0$, and $f_n \in C_0^{\infty}(\R^d)$, $f_n\ge 0$, with $f_n \rightarrow f$ in $L^s(\R^d, m)$, for each $t>0$ we have
$T_t f_n \rightarrow T_t f$ in $L^s(U, m)$ and also $P_t f_n \rightarrow P_t f$ in $L^s(U, m)$ by \eqref{rot} holds for $s\in[1,\infty)$. Thus
\begin{equation} \label{ae}
\qquad P_t f = T_t f \quad m \text{-a.e. on } U \text{ for each } t>0. 
\end{equation}
This holds for arbitrary bounded open $U$, hence also on $\R^d$. Thus $P_t f$ is an $m$-version of $T_t f$.\\ 
For $f \in L^{\infty}(\R^d, m)$, $f\ge 0$, take $f_n:=1_{B_n} \cdot f$ with $n\ge 1$. Then for each $t>0$,
\begin{equation}\label{ae2}
T_t f= \lim_{n \rightarrow \infty}T_t f_n=\lim_{n \rightarrow \infty}P_t f_n, \; m \text{-a.e.\; on } \R^d.  
\end{equation}
For each fixed $(x,t) \in V \times (\tau_3, \tau_4)$, $(P_t f_n (x) )_{n\ge 1}$ is an increasing sequence of real numbers that is bounded by one by the sub-Markovian property and continuity of $z\mapsto P_t f_n (z)$. Thus \eqref{thm main est} for $s=1$ and Lebesgue's dominated convergence theorem imply that $(P_{\cdot} f_n (\cdot ))_{n\ge 1}$ is a Cauchy sequence in 
$C^{\gamma; \frac{\gamma}{2}}(\overline{U} \times [\tau_1, \tau_2])$.
Hence we can again define 
$$
P_{\cdot} f(\cdot):=\lim_{n\to \infty}P_{\cdot} f_n(\cdot) \ \text{ in }\ C^{\gamma; \frac{\gamma}{2}}(\overline{U} \times [\tau_1, \tau_2])
$$
and \eqref{thm main est} also holds for $s=\infty$. Moreover for each $t>0$, $P_t f_n$ converges uniformly to $P_t f$ in $U$, hence in view of \eqref{ae2}, \eqref{ae} also holds for $s=\infty$. Since $U$ is an arbitrary bounded open subset in $\R^d$, we have hence shown that for any $f\in \cup_{s\in[1,\infty]} L^s(\R^d, m)$, $f\ge 0$, $P_{\cdot} f(\cdot)$ is locally parabolic H\"{o}lder continuous in $\R^d \times (0, \infty)$ and for each $t>0$, $P_t f = T_t f$ $m$-a.e. on $\R^d$. By linearity the latter is obviously also true for $f$ without the positivity assumption.
\end{proof}\\
\centerline{}
\begin{rem}\label{hkeandcomp}
(i) \eqref{thm main est} easily implies for any $s\in [1,\infty]$, $f\in L^s(\R^d, m), t>0$ (cf. for instance \eqref{rest} for $s\in [0,\infty)$ and use the sub-Markovian property for $s=\infty$) that  
\begin{eqnarray}\label{Fellerfixedtime}
\|P_{t} f \|_{C^{0,\gamma}(\overline{U})} & \leq & 2 C_6 (\tau_4-\tau_3) \|\rho \|_{L^1(V)}^{\frac{s-1}{s}} \cdot \|f\|_{L^s(\R^d, m) },
\end{eqnarray}
where $\frac{s-1}{s}:=1$ for $s=\infty$. \eqref{Fellerfixedtime} is an improvement over \eqref{tt2} in regard to analyticity, which is no more required for \eqref{Fellerfixedtime}, and in regard to the integrability order which is $s\in  [1,\infty]$ for \eqref{Fellerfixedtime} but $r\in [q\vee 2,\infty)$  for \eqref{tt2}. Only the H\"older exponent $\gamma$ in \eqref{Fellerfixedtime} depends on the domain and may vary, whereas in  \eqref{tt2} it is always $\beta$ as in \eqref{beta2}, independently of the domain. \\
Using Theorem \ref{1-3reg}, we can define $P_t(x,A)$ as in (\ref{heatker}) and we see that there exist unique sub-probability measures  $P_{t}(x,dy)$ on $(\R^d, \mathcal{B}(\R^d))$, absolutely continuous with respect to $m$ and with Radon-Nikodym derivatives $p_t(x,\cdot)$ defined by (\ref{rnd}). \\
(ii) Let $A:=(a_{ij})_{1\leq i,j \leq d}, \ \mathbf{G},\ \rho, \ \beta^{\rho,A}$, and $\mathbf{B}$ be as in Theorem \ref{1-3}, but suppose $p>d+2$ and that $m$ is a probability measure. In this case similar results to Theorem \ref{1-3reg} and the following Proposition \ref{regular2}(ii) and some additional structure with respect to duality is derived in \cite[Theorem 4.1]{BKR2}. The technique of proof is different to ours but also applies if $m$ is not restricted to be a probability measure (cf. \cite[Remark 4.2(ii)]{BKR2}). However, we insist that $K_t(x,dy)$ as occurring in \cite[Remark 4.2(ii)]{BKR2} is in contrast to what is mentioned in \cite[Remark 4.2(ii)]{BKR2} always a sub-probability measure and hence finite and moreover in case of merely locally finite measure only the  $L^1(\R^d, m)$-strong Feller property follows, whereas we derive the $L^{[1,\infty]}(\R^d,m)$-strong Feller property (see Theorem \ref{1-3reg} and Proposition \ref{regular2} for the definition), that includes the classical strong Feller property.\\
(iii) As opposed to \cite[Proposition 3.8]{AKR}, we do not need the condition $\alpha R_{\alpha}1_{\R^d}\equiv 1$  in order to derive the classical strong Feller property of $(P_t)_{t>0}$.  Also in \cite{ZhXi16}, non-explosion (see \eqref{prop1.5eq} below)  is used to obtain the classical strong Feller property. 
\end{rem}
\bigskip
Using Theorem \ref{1-3reg}, we obtain the following improvement of Proposition \ref{regular1}:
\begin{prop}\label{regular2}
Let $t,\alpha>0$ be arbitrary. Let $q, \beta$ be defined as in (\ref{q}), (\ref{beta2}), $r_{\alpha}(x,y)$ as in Remark \ref{sectorial}, and $p_t(x,y)$ as in Remark \ref{hkeandcomp}. Then under the conditions of Theorem \ref{1-3}, it holds:
\begin{itemize}
\item[(i)] $G_{\alpha}g$ has a locally H\"older continuous $m$-version of order $\beta=1 - d/p$
\begin{eqnarray}\label{tgy}
R_{\alpha}g=\int_{\R^d}g(y) R_{\alpha}(\cdot,dy)=\int_{\R^d}g(y)r_{\alpha}(\cdot,y)m(dy)
, \ \  \forall g\in \bigcup_{r\in [q,\infty]} L^r(\R^d,m).
\end{eqnarray}
In particular,  \eqref{tgy} extends by linearity  to all $g\in L^q(\R^d,m)+L^\infty(\R^d,m)$, i.e.  $(R_{\alpha})_{\alpha>0}$ is $L^{[q,\infty]}(\R^d,m)$-strong Feller.
\item[(ii)] $T_t f$ has a continuous $m$-version   
\begin{eqnarray}\label{tgy2}
P_t f= \int_{\R^d} f(y) P_t(\cdot,dy)=\int_{\R^d}f(y)p_{t}(\cdot,y)m(dy)
,\ \  \forall f\in \bigcup_{s\in [1,\infty]}L^s(\R^d,m).
\end{eqnarray}
($P_t f$ is locally H\"{o}lder continuous of order $\beta=1 - d/p$, if $f\in  \bigcup_{r\in [q,\infty)} D(L_r)$) and locally H\"{o}lder continuous with possibly changing H\"older exponents, if $f\in  \bigcup_{s\in [1,\infty]}L^s(\R^d,m)\setminus \bigcup_{r\in [q,\infty)} D(L_r)$.
In particular,  \eqref{tgy2} extends by linearity  to all $f\in L^1(\R^d,m)+L^\infty(\R^d,m)$, i.e.   $(P_{t})_{t>0}$ is $L^{[1,\infty]}(\R^d,m)$-strong Feller.
\end{itemize}
Finally, for any $\alpha>0, x\in \R^d$, $g\in L^q(\R^d,m)+L^\infty(\R^d,m)$
$$
R_{\alpha}g(x)=\int_0^{\infty} e^{-\alpha t} P_t g(x)\,dt.
$$
\end{prop}
\begin{proof}
Fix $\alpha>0, t>0, x\in \R^d$. Let $A \in \mathcal{B}(\R^d)$. Using \eqref{resker}, \eqref{reskernext}, monotone integration and \eqref{gal}, we can see that 
\begin{eqnarray}\label{bb0}
\int_{\R^d}1_A(y)r_{\alpha}(x,y)\,m(dy)=\int 1_A(y)R_{\alpha}(x,dy)=\lim_{l\to \infty}R_{\alpha}(1_{B_l\cap A})(x)=R_\alpha 1_A(x).
\end{eqnarray}
Using \eqref{heatker}, \eqref{rnd}, monotone integration and \eqref{thm main est} (cf. proof of Theorem \ref{1-3reg}) , we can see that 
\begin{eqnarray}\label{bb02}
\int_{\R^d}1_A(y) p_t(x,y) m(dy)= \int_{\R^d} 1_A(y) P_t(x,dy)=\lim_{l\to \infty}P_t 1_{B_l\cap A}(y)=P_t 1_A(x).
\end{eqnarray}
\eqref{bb0}, resp.  \eqref{bb02} extends to $g\in L^r(\R^d,m)$, $r\in [q,\infty]$, resp. $g\in L^s(\R^d,m)$, $s\in [1,\infty]$ 
in the following way. Split $g,f$ in positive and negative parts. We may hence assume that $g,f$ are positive. Then we use a monotone approximation of $g$, resp. $f$ with simple functions involving indicator functions like above, i.e. there exists an increasing sequence of simple functions $(g_n)_{n\ge 1}$ with $0\le g_n\nearrow g$, resp. $(f_n)_{n\ge 1}$ with $0\le f_n\nearrow f$. By this we can use monotone integration for the two left hand terms of \eqref{bb0}, resp.  \eqref{bb02}, and \eqref{gal}, resp.  \eqref{thm main est} for the left hand term. Thus (i) and (ii) follow.\\
The last statement follows similarly noting that for  $A \in \mathcal{B}(\R^d)$
$$
R_{\alpha}1_A=\int_0^{\infty} e^{-\alpha t} P_t 1_A\,dt
$$
$m$-a.e. hence everywhere since both sides define continuous functions and we can as before use monotone integration as well as \eqref{gal} and \eqref{thm main est} to prove the remaining assertion.
\end{proof}\\
\centerline{}
\begin{rem}\label{transfct}
Under the conditions of Theorem \ref{1-3}, we obtain analogously to \cite{AKR} that $(P_t)_{t>0}$ defined on  
$$
L^{\infty}(\R^d,m)=L^{\infty}(\R^d)\supset \mathcal{B}_b(\R^d)
$$ 
determines a (temporally homogeneous) submarkovian transition function (cf. \cite[1.2]{CW}). Thus $(P_t)_{t>0}$ satisfies condition $(\bf{H1})$ of \cite{ShTr13a}. Moreover, $P_t f$, $t>0$, is by Proposition \ref{regular2}(ii) independent of the $m$-version chosen for 
$f\in L^{\infty}(\R^d,m)$.
\end{rem}
By the results of \cite[Section 4.1]{Tr5}, the generalized Dirichlet form $\mathcal{E}$ associated with $(L_2,D(L_2))$ is strictly quasi-regular. In particular, by \cite[Theorem 6]{Tr5}  there exists a Hunt process 
$$
\tilde{\M} = (\tilde{\Omega}, \tilde{\F}, (\tilde{\F})_{t \ge 0}, (\tilde{X}_t)_{t \ge 0}, (\tilde{\P}_x)_{x \in \R^d \cup \{ \Delta \} })
$$ 
with state space $\R^d$ and life time $\tilde\zeta:=\inf\{t\ge 0\,:\,\tilde{X}_t=\Delta\}$ and cemetery $\Delta$ such that $\mathcal{E}$ is (strictly properly) associated with $\tilde{\M}$.\\
For some fixed $\varphi \in L^1(\R^d,m)_b$, $0 < \varphi \le 1$, consider the strict capacity $\text{cap}_{1,{\widehat{G}}_1\varphi}$ of $\mathcal{E}$ as defined in \cite[Definition 1]{Tr5}. Due to the properties of smooth measures with respect to $\text{cap}_{1,{\widehat{G}}_1\varphi}$ in \cite[Section 3]{Tr5} one can consider the work \cite{Tr2} with cap$_{\varphi}$ (as defined in \cite{Tr2}) replaced by $\text{cap}_{1,{\widehat{G}}_1\varphi}$. In particular \cite[Theorem 3.10 and Proposition 4.2]{Tr2} apply with respect to the strict capacity $\text{cap}_{1,{\widehat{G}}_1\varphi}$ and therefore the paths of $\tilde{\M}$ are continuous $\tilde{\P}_x$-a.s. for strictly $\mathcal{E}$-q.e. $x \in \R^d$ on the one-point-compactification $\R^d_{\Delta}$ of $\R^d$ with $\Delta$ as point at infinity. We may hence assume that 
\begin{equation}\label{contipath}
\tilde{\Omega} = \{\omega = (\omega (t))_{t \ge 0} \in C([0,\infty),\R^d_{\Delta}) \, : \, \omega(t) = \Delta \quad \forall t \ge \zeta(\omega) \}
\end{equation}
and
\[
\tilde{X}_t(\omega) = \omega(t), \quad t \ge 0.
\]
Now, we can apply the Dirichlet form method of \cite[Section 2.1.2]{ShTr13a}. There it was only developed in a symmetric setting. But here we are in the non-sectorial setting. However one can readily check  that it works nearly in the same way using Lemma \ref{algebra} instead of \cite[Lemma 2.5(i)]{ShTr13a} and modifying ${(\textbf{H2})^{\prime}}$ of \cite[Section 2.1.2]{ShTr13a} in the following way:\\ \\
${(\textbf{H2})^{\prime}}$ We can find $\{ u_n \, : \, n \ge 1 \} \subset D(L_1) \cap C_0(\R^d)$ satisfying:
\begin{itemize}
\item[(i)] For all $\varepsilon \in \mathbb{Q} \cap (0,1)$ and
$y \in D$, where $D$ is any given countable dense set in $\R^d$, there exists $n \in \N$ such that $u_n (z) \ge 1$, for all $z \in \overline{B}_{\frac{\varepsilon}{4}}(y)$ and $u_n \equiv 0$ on $\R^d \setminus B_{\frac{\varepsilon}{2}}(y)$,
\item[(ii)] $R_1\big( [(1 -L_1) u_n]^+ \big)$, $R_1\big( [(1 -L_1) u_n]^- \big)$, $R_1 \big( [(1-L_1)u_n^2]^+ \big)$, $R_1 \big( [(1-L_1)u_n^2]^- \big)$ are continuous on $\R^d$ for all $n \ge 1$,
\end{itemize}
and
\begin{itemize}
\item[(iii)] $R_1 C_0(\R^d) \subset C(\R^d)$,
\item[(iv)] For any $f \in C_0(\R^d)$ and $x \in \R^d$, the map $t \mapsto P_t f(x)$ is right-continuous on $(0,\infty)$.
\end{itemize}
We have $C_0^{2}(\mathbb{R}^d)\subset D(L_1) \cap C_0(\R^d)$ and moreover obviously $(1 -L_1) u, (1 -L_1) u^2 \in L^p(\R^d)_0$ 
for any $u\in C_0^{2}(\mathbb{R}^d)$. Consequently, by Theorem \ref{1-3reg} and Proposition \ref{regular2},  ${(\textbf{H2})^{\prime}}$  is satisfied for some countable subset of $C_0^{2}(\mathbb{R}^d)$.\\ \\
Therefore,  we obtain:
\begin{theo}\label{existhunt}
Under the conditions of Theorem \ref{1-3}, there exists a Hunt process
\[
\M =  (\Omega, \F, (\F_t)_{t \ge 0}, (X_t)_{t \ge 0}, (\P_x)_{x \in \R^d\cup \{\Delta\}}   )
\]
with state space $\R^d$ and life time 
$$
\zeta=\inf\{t\ge 0\,:\,X_t=\Delta\}=\inf\{t\ge 0\,:\,X_t\notin \R^d\}, 
$$
having the transition function $(P_t)_{t \ge 0}$ as transition semigroup, such that $\M$ has continuous sample paths in the one point compactification $\R^d_{\Delta}$ of $\R^d$ with the cemetery $\Delta$ as point at infinity.
\end{theo}
\begin{rem}\label{refinement}
Checking the details of  \cite[Section 4]{AKR} one by one with possibly only few modifications one may possibly also obtain Theorem \ref{existhunt}. 
\end{rem}
\begin{lem}\label{fini1}
Let $\E_x$ denote the expectation with respect to $\P_x$, $x\in \R^d$, where $\M$ is as in  Theorem \ref{existhunt}. Then:
\begin{itemize}
\item[(i)] For any $x\in \R^d$, $\alpha>0, t>0$, we have
\[
R_{\alpha}g(x)=\int_{\R^d}r_{\alpha}(x,y)g(y)m(dy)=\E_x \left [\int_0^\infty e^{-\alpha s} g(X_s)ds\right ],  
\]
for any $g\in L^q(\R^d,m)+L^\infty(\R^d,m)$, and
\[
P_t f(x)=\int_{\R^d}p_{t}(x,y)f(y)m(dy)=\E_x\left [f(X_t)\right ], 
\]
for any $f\in L^1(\R^d,m)+L^\infty(\R^d,m)$. \\
In particular, integrals of the form 
$\int_0^\infty e^{-\alpha s} h(X_s)ds$, $\int_0^t h(X_s)ds$, $t\ge 0$ are for any $x\in \R^d$, whenever they are well-defined,  $\P_x$-a.s.  independent of the measurable $m$-version chosen for $h$. 
\item[(ii)] Let $g\in L^r(\R^d,m)$ for some $r\in [q,\infty]$. Then for any ball $B$ there exists a constant $c_{B,r}$, depending in particular on $B$ and $r$, such that for all $t \ge  0$
\begin{equation} \label{resest}
\sup_{x\in \overline{B}}\E_x\left [ \int_0^t |g|(X_s) \, ds \right ] \le  e^t c_{B,r} \|g\|_{L^r(\R^d, m)}.
\end{equation}
In particular, by the local boundedness of $\rho$, for any $f\in L^{q}(\R^d)_0$, and any open ball $B$, \eqref{resest} implies
$$
\sup_{x\in \overline{B}}\E_x\left [ \int_0^t |f|(X_s) \, ds \right ] \le c\, \|f\|_{L^q(\R^d)},
$$
where the constant $c$ depends only on $t,q, B$ and the support of $f$.
\item[(iii)] Let $u \in D(L_r)$, for some $r \in [q, \infty)$ and $\alpha > 0$, $t>0$. Then for any $x \in\R^d$
\[
R_{\alpha} \big( (\alpha - L_r) u\big)(x) = u(x),
\]
and 
\[
P_t u(x) - u(x) = \int_0^t P_s(L_ru)(x) \, ds.
\]
\end{itemize}
\end{lem}
\begin{proof}
(i) By Remark \ref{transfct} and Theorem \ref{existhunt}, we have for any $t> 0$, $x\in \R^d$, $h\in L^{\infty}(\R^d,m)$
\begin{eqnarray}\label{process}
P_t h(x)  = \int_{\R^d}p_{t}(x,y)h(y)m(dy)=\E_x\left [h(X_t)\right ], 
\end{eqnarray}
and the expressions in (\ref{process}) are all well-defined, i.e. do not change in value for any $m$-version of $h$. Now the resolvent and semigroup representations follow by splitting functions in $g\in \bigcup_{r\in [q,\infty]} L^r(\R^d,m)$ and $f\in\bigcup_{s\in [1,\infty]} L^s(\R^d,m)$ into their positive and negative parts, using monotone approximations of these with functions in $L^{\infty}(\R^d,m)$ and finally linearity, which is possible since all expressions are finite by Proposition \ref{regular2}. In particular, the limits will as the original expressions in (\ref{process}) also not depend on the chosen $m$-versions, which concludes the proof.\\
(ii) Using in particular (i) and \eqref{gal2}, we get
\begin{eqnarray*}
\sup_{x \in \overline{B}} \E_x\left [ \int_0^t |g|(X_s) \, ds \right ] &\leq&  e^t\, \sup_{x \in \overline{B}}\, \E_x\left [ \int_0^{\infty}  e^{-s}|g|(X_s)  \,ds \right ] \\
&= &e^t \sup_{x \in \overline{B}}\, R_1 |g|(x) \leq e^t c_{B} \|g\|_{L^r(\R^d, m)}.
\end{eqnarray*}
Using (i), the proof of (iii) works exactly as in \cite[Lemma 5.1]{AKR}. However, we emphasize that due to the increased regularity  $r\ge q$ from (i) (coming from Proposition \ref{regular1}) in comparison to  $r\ge p$ in \cite{AKR}, we obtain more general statements in (ii) and (iii).
\end{proof}\\
For $A\in \mathcal{B}(\R^d)$, define
$$
\sigma_{A}:=\inf\{t>0\,:\, X_t\in A\}
$$
and 
$$
\sigma_n:=\sigma_{\R^d\setminus B_n},n\ge 1.
$$
\begin{lem}\label{nest}
Let $\M$ be as in  Theorem \ref{existhunt}. Then:
\begin{itemize}
\item[(i)] For any $x \in \R^d$, we have
\[
\P_x \Big(\lim_{n \rightarrow \infty} \sigma_{n} \ge \zeta \Big)=1.
\] 
\item[(ii)] For any $x\in \R^d$, $t\ge 0$, we have
\[
\P_x\left (\int_0^t |f|(X_s)ds<\infty\right )=1,  \text{ if }\ \ f\in \bigcup_{r \in [q, \infty]} L^r(\R^d,m)
\]
and
\[
\P_x \left (\left \{\int_0^t |f|(X_s)ds<\infty\right \} \cap \{t<\zeta\}\right )=\P_x \left (\{t<\zeta\}\right ), \text{ if } \  \ f\in L^q_{loc}(\R^d,m).
\]
\end{itemize}

\end{lem}
\begin{proof}
(i) By Proposition \ref{regular2} and Lemma \ref{fini1}(i), we have that $\E_{\cdot} \left [\int_0^\infty e^{-\alpha s} g(X_s)ds\right ]$ is an $m$-version of $G_{\alpha}g$, for all $\alpha>0$ and $g\in L^{\infty}(\R^d,m)$. It hence follows by \cite[IV. Theorem 3.1]{St2} (or \cite[Proposition 2(ii)]{Tr5}) that $\mathcal{E}$ is quasi-regular. Therefore by \cite[IV. Definition 1.7]{St2} there exists an $\mathcal{E}$-nest $(E_k)_{k\ge 1}$ of compact subsets of $\R^d$. Then \cite[IV. Lemma 3.10]{St2} implies, 
 $\P_x \Big(\lim_{k \rightarrow \infty} \sigma_{\R^d \setminus E_{k}} \ge \zeta \Big)=1$ for $\mathcal{E}$-q.e. $x\in \R^d$, hence in particular for $m$-a.e. $x\in \R^d$ by \cite[III. Remark 2.6]{St2}. Since $(B_n)_{n\ge 1}$ is an open cover of $E_k$ for each $k$, and $\sigma_A\le \sigma_B$ whenever $B\subset A$, we then obtain $\P_x \Big(\lim_{n \rightarrow \infty} \sigma_{n} \ge \zeta \Big)=1$ for $m$-a.e. $x\in \R^d$. Now the result follows exactly as in \cite[Lemma 3.3]{RoShTr}.\\
(ii) The first statement immediately follows from Lemma \ref{fini1}(ii). For the second statement it is enough to show that for any $t\ge 0$ and $x\in \R^d$
\begin{eqnarray}\label{welldefsde}
\P_x \left (1_{\{t<\zeta\}}\int_0^t |f|(X_s)ds<\infty \right )=1, \text{ if } \  \ f\in L^q_{loc}(\R^d,m).
\end{eqnarray}
It holds $\P_x(n\wedge \sigma_{n}<\zeta)=1$ for any $n\ge 1$ and $x\in \R^d$, since $\M$ has continuous sample paths on the one-point-compactification $\R^d_{\Delta}$. Thus using (i), we get that the left hand side of (\ref{welldefsde}) equals
\begin{eqnarray}\label{welldefsde2}
\lim_{n\to \infty} \P_x \left (1_{\left \{t<n\wedge \sigma_{n}\right \}}\int_0^t |f|(X_s)ds<\infty \right ).
\end{eqnarray}
Now, fix $x\in \R^d$. Then there exists $N_0\in \N$ with $x\in B_n$ for any $n\ge N_0$. Consequently, for any $n\ge N_0$ we have $\P_x$-a.s. that $X_s\in B_n$ for any $s\in [0,t]$, if $t<\sigma_{n}$. It follows with the help of 
Lemma \ref{fini1}(ii)
\[
\E_x \Big [1_{\left \{t<n\wedge \sigma_{n}\right \}}\int_0^t |f|(X_s)ds \Big ]\le 
\E_x \left [\int_0^t |f|1_{B_n}(X_s)ds \right ]<\infty, \ \ \ \forall n\ge N_0.
\]
Thus each sequence member in (\ref{welldefsde2}) is equal to one and therefore (\ref{welldefsde}) holds.
\end{proof}
\begin{prop}\label{mart}
Let $\M$ be as in  Theorem \ref{existhunt}. 
Let $u \in D(L_r)$, for some $r \in [q, \infty)$. Then
\[
M_t ^u: = u(X_t) - u(x) - \int_0^t L_r u(X_s) \, ds , \quad t \ge 0,
\] 
is  a continuous $(\mathcal{F}_t)_{t \ge 0}$-martingale under $\P_x$  for any $x \in \R^d$. If $r\ge 2q$, then $M^u$ is square integrable. 
\end{prop}
\begin{proof}
The first result is an immediate consequence of Lemma \ref{fini1} (see for instance \cite[Chapter 7, (1.6) Theorem]{d}).
The second follows from Lemma \ref{fini1}(i) and (ii).
\end{proof}
\begin{prop}\label{quadvar}
Let $\M$ be as in  Theorem \ref{existhunt}. Let $u \in C_0^2(\R^d)$, $t\ge 0$. Then the quadratic variation process $\langle M^u \rangle$ of the continuous martingale $M^u$ satisfies $\P_x$-a.s for any $x \in \R^d$, $t\ge 0$
$$
\langle M^u \rangle_t=\int_0^t \langle A\nabla u, \nabla u\rangle(X_s)ds.
$$
In particular, by Lemma \ref{fini1}(ii) $\langle M^u \rangle_t$ is $\P_x$-integrable for any $x \in \R^d$, $t\ge 0$  and  so $M^u$ is square integrable.
\end{prop}
\begin{proof}
For $g\in C_0^2(\R^d)$, we have  $g\in D(L_r)$ and $L_1g=L_r g$ for any $r\in [1,p]$. Thus for $u\in  C_0^2(\R^d)$, we get by Proposition \ref{mart} and Lemma \ref{algebra}
\begin{eqnarray*}
u^2(X_t) - u^2(x)= M_t^{u^2} +\int_0^t \left (\langle A\nabla u,\nabla u \rangle(X_s) + 2 u L_1 u(X_s)\right )\,ds.
\end{eqnarray*}
Applying It\^o's formula to the continuous semimartingale $(u(X_t))_{t\ge 0}$, we obtain
\begin{eqnarray*}
u^2(X_t) - u^2(x)= \int_0^t 2u(X_s)dM_s^{u} +\int_0^t 2 u L_r u(X_s)\,ds + \langle M^u \rangle_t.
\end{eqnarray*}
The last two equalities imply that $\left (\langle M^u \rangle_t-\int_0^t \langle A\nabla u, \nabla u\rangle(X_s)ds\right )_{t\ge 0}$ is 
a continuous $\P_x$-martingale of bounded variation for any $x\in \R^d$. This implies the assertion.
\end{proof}\\
For the following result, see for instance \cite[Theorem 1.1, Lemma 2.1]{ChHu}, that we can apply locally.
\begin{lem}\label{sigmaprop}
Under the assumptions of Theorem \ref{1-3} on the diffusion matrix $A$, there exists a matrix  of functions $\sigma=(\sigma_{ij})_{1\le i,j\le d}$ with $\sigma_{ij}\in C(\R^d)$ for all $i,j$  such that
\begin{equation*}
A(x)=\sigma\sigma^T(x), \ \  \forall x\in \R^d, 
\end{equation*} 
i.e. 
$$
a_{ij}(x)=\sum_{k=1}^d \sigma_{ik}(x)\sigma_{jk}(x), \ \  \forall x\in \R^d, \ 1\le i,j\le d
$$
and 
$$
\det(\sigma(x))>0, \ \  \ \forall x\in \R^d.
$$
\end{lem}

\begin{theo}\label{weakexistence}
Let $A:=(a_{ij})_{1\leq i,j \leq d}$, $\mathbf{G}$, be as in Theorem \ref{1-3}. 
Consider the Hunt process $\M$ from Theorem \ref{existhunt} with coordinates $X_t=(X_t^1,\ldots,X_t^d)$ and suppose that  $\M$ is non-explosive, i.e.
$$
\P_x(\zeta=\infty)=1\ \ \text{for any } x\in \R^d.
$$
\begin{itemize}
\item[(i)] Let $(\sigma_{ij})_{1 \le i,j \le d}$ be as in Lemma \ref{sigmaprop}. Then it holds $\P_x$-a.s. for any $x=(x_1,\ldots,x_d)\in \R^d$, $i=1,\dots,d$
\begin{equation}\label{weaksolution}
X_t^i = x_i+ \sum_{j=1}^d \int_0^t \sigma_{ij} (X_s) \, dW_s^j +   \int^{t}_{0}   g_i(X_s) \, ds, \quad 0\le  t <\infty,
\end{equation}
where $W = (W^1,\dots,W^d)$ is a standard $d$-dimensional Brownian motion starting from zero.
\item[(ii)] Let $(\sigma_{ij})_{1 \le i \le d,1\le j \le m}$, $m\in \N$ arbitrary but fixed, be any matrix consisting of continuous functions  $\sigma_{ij}\in C(\R^d)$ for all $i,j$,  such that $A=\sigma\sigma^T$ (where $A$ satisfies the assumptions of Theorem \ref{1-3}), i.e. 
$$
a_{ij}(x)=\sum_{k=1}^m \sigma_{ik}(x)\sigma_{jk}(x), \ \  \forall x\in \R^d, \ 1\le i,j\le d.
$$
Then on a standard extension 
of $(\Omega, \F, (\F_t)_{t\ge 0}, \P_x )$, $x\in \R^d$, that we denote for notational convenience again 
by $(\Omega, \F, (\F_t)_{t\ge 0}, \P_x )$, $x\in \R^d$, there exists a standard  $m$-dimensional Brownian motion $W = (W^1,\dots,W^m)$ starting from zero such that (\ref{weaksolution}) holds with  
$\sum_{j=1}^d$ replaced by $\sum_{j=1}^m$.
\end{itemize}

\end{theo}
\begin{proof}
(i) Consider the stopping times
$$
D_n:=D_{\R^d\setminus B_n}:=\inf\{t\ge0\,:\, X_t\in \R^d\setminus B_n\} \ \ \ n\ge 1.
$$
Since $\M$ is non-explosive, it follows from  Lemma \ref{nest}(i) that $D_n\nearrow \infty$ $\P_x$-a.s. for any $x\in \R^d$.  Let $v\in C^{2}(\R^d)$. Then we claim that
\[
M_t ^v: = v(X_t) - v(x) - \int_0^t \left (\frac{1}{2}\sum_{i,j=1}^{d}a_{ij}\partial_i\partial_j v+\sum_{i=1}^{d}g_i\partial_i v\right )(X_s) \, ds , \quad t \ge 0,
\]
is a continuous square integrable local $\P_x$-martingale with respect to the stopping times $(D_n)_{n\ge 1}$ for any $x\in \R^d$. Indeed, let $(v_n)_{n\ge 1}\subset C_0^2(\R^d)$ be such that $v_n=v$ pointwise on $\overline{B}_n$, $n\ge 1$. 
Then for any $n\ge 1$, we have $\P_x$-a.s
$$
M_{t\wedge D_n}^v=M_{t\wedge D_n}^{v_n}, \ \ t\ge 0,
$$
and $(M_{t\wedge D_n}^{v_n})_{t\ge 0}$ is a square integrable $\P_x$-martingale for any $x\in \R^d$ by Proposition \ref{quadvar}.
Now let $u_i \in C^{2}(\R^d)$, $i=1,\dots,d$, be the coordinate projections, i.e. $u_i(x)=x_i$. Then by Proposition \ref{quadvar}, polarization and localization with respect to $(D_n)_{n\ge 1}$, the quadratic covariation processes satisfy
\[
\langle M^{u_i}, M^{u_j} \rangle_t =  \int_0^t  a_{ij}(X_s) \, ds, \quad 1 \le i,j  \le d, \ t \ge 0.
\]
Using Lemma \ref{sigmaprop} we obtain by \cite[II. Theorem 7.1]{IW89} that there exists a $d$-dimensional Brownian motion $(W_t)_{t \ge 0} = (W_t^1,\dots, W_t^d)_{t \ge 0}$ on $(\Omega, \F, (\F_t)_{t\ge 0}, \P_x )$, $x\in \R^d$,  such that
\begin{eqnarray}\label{sigma1}
M_t^{u_i} = \sum_{j=1}^{d}  \int_0^t \ \sigma_{ij} (X_s) \ dW_s^j, \quad 1 \le i \le d, \  t\ge 0.
\end{eqnarray}
Since for any $x\in \R^d$, $\P_x$-a.s.
\begin{eqnarray}\label{drift1}
M_t^{u_i}= X_t^i - x_i - \int_0^t g_i(X_s) \, ds , \quad t \ge 0,
\end{eqnarray}
the assertion follows.\\
(ii) The proof of (ii) is similar to the proof of (i) but uses \cite[II. Theorem 7.1']{IW89} instead of \cite[II. Theorem 7.1]{IW89} (see \cite[IV. Proposition 2.1]{IW89})
\end{proof}
\begin{rem}\label{weakexistence2}
Theorem \ref{weakexistence} holds in general only up to $\zeta$, when one does not impose non-explosion. Here, we only sketch in detail the proof in case of Theorem \ref{weakexistence}(i). (The case of Theorem \ref{weakexistence}(ii) is nearly the same but one has to work on a standard extension of the underlying probability space).
One first uses that for $v_k\in C_0^2(\R^d)$, $1\le k\le d$, one has by Proposition \ref{quadvar}
\[
\langle M^{v_k}, M^{v_l} \rangle_t =  \int_0^t  \Phi_{kl}(X_s) \, ds, \quad 1 \le k,l  \le d, \ t \ge 0,
\]
where $\Phi_{kl}=\sum_{i,j=1}^d a_{ij}\partial_j v_k\partial_i v_l$, so that 
$$
\Phi_{kl}=\sum_{m=1}^d\Psi_{km}\Psi_{lm},  \ \text{ with }\ \ \Psi_{km}=\sum_{i=1}^d\sigma_{im}\partial_i v_k,\ \  
 1\le k,l,m \le d.
$$
Note that we then do no longer have 
\begin{eqnarray}\label{sigma2}
\det((\Psi_{km})_{1\le k,m \le d})\not= 0
\end{eqnarray}
globally as opposed to Lemma \ref{sigmaprop}. However, choosing $v_k(x)=v_k^n(x)=x_k$ on $\overline{B}_n$, $1\le k\le d$, $n\ge 1$, we can obtain (\ref{sigma2}) locally on $B_n$, hence (\ref{sigma1}) locally on $\{t\le D_n\}$ for each $n\ge 1$. Consequently, we also get (\ref{drift1}) locally on $\{t \le D_n\}$ for each $n\ge 1$. Then showing consistency of the local martingale and drift parts, we obtain (\ref{weaksolution}) up to $\zeta$ by Lemma \ref{nest}(i).\\
\end{rem}
\section{Long time behavior, moment inequalities and \\ uniqueness of invariant probability measures}\label{4}
In this section we investigate long time behavior like non-explosion, recurrence and ergodicity. We will also investigate some moment inequalities that are well-known for classical  It\^o equations with continuous coefficients. We saw in Theorem \ref{weakexistence} and Remark \ref{weakexistence2} that we can obtain a weak solution up to the life time $\zeta$. We first provide explicit non-explosion criteria, i.e. explicit criteria that imply the assumption 
$$
\P_x(\zeta=\infty)=1\ \ \text{for any } x\in \R^d
$$
of Theorem  \ref{weakexistence}. 
\subsection{Non-explosion criteria and moment inequalities}
\subsubsection{Non-explosion criteria and moment inequalities without involving the density $\rho$}
In this subsection we consider non-explosion criteria that only depend on the coefficients of the underlying SDE. We first derive a lemma that is a variant of the construction in \cite[page 197]{BKRS} and then a non-explosion criterion by following a probabilistic technique which traces back at least to \cite[10.2]{StrVar}.
\begin{lem}\label{lyapunovconstr}
Let $f\in C^2(\R^d)$ be a positive, strictly increasing and unbounded radial function, i.e. $f\ge 0$ pointwise, $f(x)\equiv c_r$ on $\partial B_r$ with $0<c_r<c_{r'}$ whenever $0<r<r'$, and $c_n\to \infty$ as $n\to \infty$. Under the assumptions on the coefficients of Theorem \ref{1-3}, suppose that there exist $M>0$, $N_0\in \N$ such that 
$$
Lf=\frac{1}{2}\sum_{i,j=1}^{d}a_{ij}\partial_i\partial_j f+\sum_{i=1}^{d}g_i\partial_i f\le M f \quad \text{a.e. on } \R^d\setminus B_{N_0}.
$$ 
Let $\phi\in C^2(\R)$, such that $\phi, \phi'\ge 0$ pointwise, 
$$
\phi(t)=\begin{cases} \ \sup_{B_{N_0}}f&  \text{ if } t\le \sup_{B_{N_0}}f,\\
\ t\quad& \text{ if } t\ge \sup_{B_{N_0+1}}f,
\end{cases}
$$
and let for arbitrary $\alpha \ge 0$
$$
\psi:=\phi\circ f+C_{\phi,A}+\alpha,
$$
where 
$$
C_{\phi,A}:=M\big (  c_{\phi} \sup_{B_{N_0+1}}f +\frac{c_{\phi}}{2M}\sup_{B_{N_0+1}}\langle A \nabla f, \nabla f \rangle \big )
$$
and
$$
c_{\phi}:=\sup_{B_{N_0+1}\setminus\overline{B}_{N_0}}\phi'\circ f +\sup_{B_{N_0+1}\setminus\overline{B}_{N_0}}|\phi''\circ f|.
$$
Then $\psi\in C^2(\R^d)$, $\psi > 0$ pointwise, $\inf_{\partial B_n}\psi \nearrow \infty$ as $n\to \infty$, $n\ge N_0$, and
$$
L\psi \le M \psi \quad \text{a.e. on }  \R^d.
$$ 
\end{lem}
\begin{proof}
Using the formula
$$
L(\phi(f))=\phi'(f)Lf+\frac{1}{2}\phi''(f)\langle A \nabla f, \nabla f \rangle
$$
the assertion is easily verified.
\end{proof}
\begin{theo}\label{nonextheo}
Suppose that  (\ref{conservative2}) holds. Then for $\M$ as in  Theorem \ref{existhunt} it holds
$$
\P_x(\zeta=\infty)=1\ \ \text{for any } x\in \R^d.
$$
\end{theo}
\begin{proof}
We first show the statement corresponding to (\ref{conservative2}). Let $u_n \in C_0^2(\R^d)$, $n\ge 1$, be positive functions such that
$$
u_n(x)=\begin{cases} \ \|x\|^2\quad  \text{ if } x\in \overline{B}_n,\\
\ 0\quad\quad \text{ if } x\in \R^d\setminus B_{n+1}.
\end{cases}
$$
Then by Proposition \ref{mart}
$$
Y_t^n:=u_n(X_t), \ t\ge 0
$$
is a positive continuous $\P_x$-semimartingale for any $x\in \R^d$, $n\ge 1$.\\
Let $f(x)=\ln(\|x\|^2+1)+1$, $x\in \R^d$ and let $\psi$, $\phi$ and $C_{\phi,A}$ be as in Lemma \ref{lyapunovconstr} with $\alpha=0$.
By It\^o's formula applied to $Y^n$ with the function $e^{-Mt}\varphi(y)$, 
$$
\varphi(y):=\phi(\ln(1+y)+1)+C_{\phi,A},
$$
we obtain $\P_x$-a.s. for any $x\in B_n$
$$
e^{-M t} \varphi(Y_t^n)=\varphi(Y^n_0)+ \int_0^t e^{-M s} \varphi'(Y^n_s)dM_s^{u_n} +
\int_0^t e^{-M s}  (L-M)(\varphi\circ u_n)(X_s) \, ds.
$$
Note that $(L-M)(\varphi\circ u_n)=(L-M)\psi\le 0$ $m$-a.e. on $\overline{B}_n$ for each $n\ge 1$. Therefore, using the last part of Lemma \ref{fini1}(i), we can see that
$$
e^{-M t\wedge \sigma_{n}}\varphi\circ u_n(X_{t\wedge \sigma_{n}}), \ t\geq 0
$$ 
is a positive continuous $\P_x$-supermartingale for any $x\in B_n$, $n\ge 1$. Since $\M$ has continuous sample paths on the one-point-compactification $\R^d_{\Delta}$, we have that $\|X_{t\wedge \sigma_{n}}\|=n$  $\P_{x}$-a.s. on $\{\sigma_{n}\le t\}$ for any $x\in B_n$. Now let $x\in \R^d$ be arbitrary.
Then $x\in B_{k_0}$ for some $k_0\in \N$ and since supermartingales have decreasing expectations, we get for any $n> k_0$
\begin{eqnarray*}
\phi\left (\ln(\|x\|^2+1)+1\right )+C_{\phi,A}& = & \E_x[\varphi \circ u_n(X_0)]\\
&\ge& \E_x[e^{-M t\wedge \sigma_{n}}\varphi\circ u_n(X_{t\wedge \sigma_{n}})]\\
& \ge &e^{-M t}\E_x[\varphi\circ u_n(X_{t\wedge \sigma_{n}})1_{\{\sigma_{n}\le t\}}]\\
& \ge& e^{-M t}\left (\phi\left (\ln(n^2+1)+1\right )+C_{\phi,A}\right )\P_x(\sigma_{n}\le t).
\end{eqnarray*}
Consequently
$$
\P_x(\zeta\le t)=\lim_{n\to \infty}\P_x(\sigma_{n}\le t)=0
$$
for any $t\ge 0$, which implies the assertion. 
\end{proof}\\

\begin{rem}\label{consstannat}
(i) Suppose that for the semigroup $(T_t)_{t>0}$ defined on $L^{\infty}(\R^d,m)$ it holds
\begin{eqnarray}\label{conservative}
T_t 1_{\R^d}=1 \ m\text{-a.e. for some (and hence all)} \  t>0.
\end{eqnarray}
Then, since $T_t 1_{\R^d}=P_t 1_{\R^d}$ $m$-a.e. and $P_t 1_{\R^d}$ is continuous by the strong Feller property  (cf. Proposition \ref{regular2}(ii))
\begin{eqnarray}\label{prop1.5eq}
P_t 1_{\R^d}(x) = 1 \text{ for any } x \in \R^d, t>0, \text{ or equivalently } \M \text{ is non-explosive.}
\end{eqnarray}
(ii) Using (i),  the non-explosion criterion  (\ref{conservative2})  can be recovered form the dual version of \cite[Proposition 1.10]{St99}. Indeed, (\ref{conservative}) holds, if and only if $m$ is invariant for the $L^1 (\R^d, m)$-semigroup $(\widehat{T}_t)_{t>0}$. Then Theorem \ref{nonextheo} follows by applying the dual version of \cite[Proposition 1.10(b)]{St99} to the $C^2$-function $\psi$ as defined in the proof of Theorem \ref{nonextheo} and then using \eqref{prop1.5eq}.\\ 
(iii) In general, $\M$ will be non-explosive whenever there exists $\psi\in C^2(\R^d)$ and $M>0$, such that $\inf_{\partial B_n}\psi \to \infty$ as $n\to \infty$ and $L\psi\le M\psi$ a.e. on $\R^d$. This follows from \cite[Proposition 1.10]{St99} and (i), and can be shown as well by applying the technique of supermartingales from Theorem \ref{nonextheo}, using a generalized version of Lemma \ref{lyapunovconstr} (see \cite[page 197]{BKRS}), and noting that $(M_{t\wedge D_n}^v)_{t\ge 0}$,
is a martingale for any $v\in C^2(\R^d)$ (see proof of Theorem \ref{weakexistence}(i)).
Note the subtle difference that  \cite[Proposition 1.10]{St99} is proved by analytic means (starting from the $L^1$-generator or $L^1$-semigroup) and only leads to (\ref{conservative}), whereas Theorem \ref{nonextheo} is proven by probabilistic means (starting from Proposition \ref{mart}) and directly leads to \eqref{prop1.5eq} regardless of the classical strong Feller property. 
\end{rem}

\begin{theo}\label{supestimate}
\begin{itemize}
\item [(i)] Assume for some $N_0\in \N$ and some $p>0$, there exists $M>0$ such that
\begin{eqnarray}\label{conservative**}
\left(\frac{p-2}{2} \right) \frac{\langle A(x)x, x \rangle}{ \left \| x \right \|^2 +1}+ \frac12\mathrm{trace}A(x)+ \big \langle \mathbf{G}(x), x \big \rangle \leq M\left ( \left \| x \right \|^2+1\right ),
\end{eqnarray}
for a.e. $x\in \R^d\setminus B_{N_0}$.
Then $\M$  of Theorem \ref{existhunt} is non-explosive and for any open ball $B$ there exists a constant $C_B>0$, such that 
\[
\sup_{x\in \overline{B}}\E_{x}\left[\|X_t \|^p\right] \leq C_B \cdot e^{M\cdot t}, \quad \forall t \geq 0.
\]
\item [(ii)] Let $\sigma=(\sigma_{ij})_{1\le i,j\le d}$ be as in Lemma \ref{sigmaprop} and and $\mathbf{G}$ as in Theorem \ref{1-3}. Assume that for some $N_0\in \N$ and $C_1>0$
\begin{equation} \label{lineargrowth}
\max_{1\leq i,j \leq d}|\sigma_{ij}(x)|+\max_{1\leq i\leq d}|g_i(x)| \leq C_1(\|x\|+1) \quad \text{for a.e. $x \in \R^d \setminus B_{N_0}$}.
\end{equation}
Then $\M$ of Theorem \ref{existhunt}  is non-explosive and for any $T>0$, and any open ball $B$, there exist constants $C_{T,B}$, $C_T$ such that
\[
\sup_{x \in \overline{B}}\E_{x}\left[\sup_{s \leq t} \|X_s \|^2\right] \leq C_{T,B} \cdot e^{C_T \cdot t}, \quad \forall t\le T.
\]
\end{itemize}
\end{theo}
\begin{proof}
 (i) Let $f(x)=(\|x\|^2+1)^{\frac{p}{2}}$. Then \eqref{conservative**} implies $Lf(x) \leq Mp \cdot f(x)$ for a.e. $x\in \R^d\setminus B_{N_0}$. Let $\phi, \psi$, and $C_{\phi,A}$  be as in Lemma \ref{lyapunovconstr} with $\alpha:=\sup_{B_{N_0+1}}f$.\\
Let $\varphi(y):=\phi((y+1)^{\frac{p}{2}})+C_{\phi,A}+\alpha$. Applying It\^o's formula to $u_n(X_{\cdot})$, where $u_n$ is as in the proof of Theorem \ref{nonextheo}, with the function $e^{-Mp\cdot t}\varphi(y)$, we obtain exactly as in the proof of Theorem \ref{nonextheo} that $\M$ is non-explosive. For arbitrary $n\in \N$ and $x\in B_n$ it holds 
\begin{eqnarray*}
\big ( C_{\phi,A}+ 2\sup_{B_{N_0+1}} f  \big ) f(x) \ge \psi(x) 
\ge \E_x[e^{-(M\cdot p) t\wedge \sigma_{n}}\varphi\circ u_n(X_{t\wedge \sigma_{n}})].
\end{eqnarray*}
Using $f\le \psi$ pointwise, $\sigma_n\nearrow \infty$, Fatou's lemma and the previous inequality, we get
\begin{eqnarray*}
e^{-Mp\cdot t} \E_x[f(X_{t})]  \leq  \liminf_{n\rightarrow \infty} \E_x[ e^{-(M\cdot p) t\wedge \sigma_{n}}\varphi\circ u_n(X_{t\wedge \sigma_{n}})]  \leq \big ( C_{\phi,A}+ 2\sup_{B_{N_0+1}}f \big ) f(x).
\end{eqnarray*}
Thus,
\begin{eqnarray*}
\E_x[\|X_{t}\|^p] \leq \underbrace{\big ( C_{\phi,A}+ 2\sup_{B_{N_0+1}}f \big )  (\|x\|^2+1)^{\frac{p}{2}}}_{=:C_x}\; e^{Mp\cdot t}.
\end{eqnarray*}
Now set $C_B:=\sup_{x\in \overline{B}}C_x$.\\ \\
(ii) \eqref{lineargrowth} implies 
$$
\text{trace}(A(x)) = \sum_{i,j=1}^{d} \sigma_{ij}(x)^2 \leq 2d^2 C_1^2(\|x\|^2+1) \quad \text{  for a.e. $x \in \R^d \setminus B_{N_0}$}
$$
and
$$
\big \langle \mathbf{G}(x), x \big \rangle \leq \left( \sum_{i=1}^{d} g_i(x)^2 \right)^{\frac{1}{2}} \|x\| \leq 2dC_1 (\|x\|^2+1) \quad \text{  for a.e. $x \in \R^d \setminus B_{N_0}$}.
$$
Thus \eqref{conservative2} holds, so that $\M$ is non-explosive by Theorem \ref{nonextheo} and \eqref{weaksolution} holds. Consequently, $\P_{x}$-a.s. for any $1\le i\le d$
\begin{eqnarray} \label{sdeest1} 
&& \sup_{0 \leq s \leq t \wedge \sigma_{n}} |X_s^i|^2 \nonumber \\
&\leq& (d+2) \left(x_i^2  + \sum_{j=1}^d \sup_{\;0 \leq s \leq t \wedge \sigma_{n} } \left |\int_0^{s} \sigma_{ij} (X_u) \, dW_u^j \right |^2 + t\int^{ t \wedge \sigma_{n} }_{0}   |g_i(X_u)|^2 \, du \right).
\end{eqnarray}
Note that $\sum_{i,j=1}^{d} \sigma_{ij}(x)^2 =\text{trace}(A(x)) \leq d \cdot \Lambda_{B_{N_0}} \leq d\cdot  \Lambda_{B_{N_0}}(\|x\|^2+1)$ for a.e. $x \in B_{N_0}$. 
Thus
\begin{eqnarray} \label{sdeest2}
&& \sum_{i,j=1}^{d} \E_x \left [ \sup_{\;0 \leq s \leq t \wedge \sigma_{n} } \left |\int_0^{s} \sigma_{ij} (X_u) \, dW_u^j \right |^2   \right ] \nonumber \\
&\leq&   4 \E_x \left [\sum_{i,j=1}^{d}  \left< \int_0^{\cdot} \sigma_{ij} (X_{u} ) \, dW_u^j \right>_{t\wedge  \sigma_{n}}  \;\right ] \nonumber \\
&\leq& 4 \E_x \left [  \sum_{i,j=1}^{d}  \int_0^{t \wedge  \sigma_{n}  } \sigma_{ij}^2(X_{u})  du   \right ] \nonumber \\
&\leq& \underbrace{4\left( 2d^2 C_1^2+d \Lambda_{B_{N_0}} \right)}_{=:C_2} \E_{x} \left[\int_0^{t\wedge  \sigma_{n} } (\|X_{u}\|^2+1)\;du\right] \nonumber \\
&\leq& C_2\; \int_0^{t}\E_{x} \left[ \sup_{0 \leq s \leq u \wedge \sigma_{n} }\|X_{s}\|^2 \right]du +C_2 T.
\end{eqnarray}
Now let $x \in \overline{B}$, and $\forall t\le T$. Then using \eqref{resest}, \eqref{lineargrowth},  for any $n\in \N$ and $1\le i\le d$
\begin{eqnarray} \label{sdeest3}
&& \E_{x}\left[\int^{ t \wedge \sigma_{n} }_{0}   |g_i(X_u)|^2 \, du \right] \nonumber \\
&\leq& \E_{x}\left[\int^{ T }_{0}   |g_i 1_{B_{N_0}}|^2(X_u) \, du \right]+\E_{x}\left[\int^{ t \wedge  \sigma_{n} }_{0}    |g_i 1_{\R^d \setminus B_{N_0}}|^2(X_{u })  \, du \right] \nonumber \\
&\leq& c_{B,p} e^T \left \|g_i 1_{B_{N_0}} \right \|^2_{L^{p}(\R^d, m)}+ 2C_1 \E_{x} \left[ \int_{0}^{t\wedge  \sigma_{n} } (\|X_{u}\|^2+1)\; du \;\right] \nonumber  \\
&\leq& c_{B,p} e^T \sup_{\overline{B}_{N_0}}|\rho|^{\frac{2}{p}}\cdot \|g_i\|^2_{L^{p}(B_{N_0})}+2C_1 \int_{0}^{t}\E_{x} \left[  \sup_{0 \leq s \leq u\wedge \sigma_{n}}\|X_{s}\|^2 \right] du + 2C_1T.
\end{eqnarray}
Now let $h_n(t):=\E_{x} \left[  \sup_{0 \leq u \leq t\wedge \sigma_{n}}\|X_{u}\|^2 \right]$. Then by \eqref{sdeest1}, \eqref{sdeest2}, \eqref{sdeest3}, we obtain
\begin{eqnarray*}
h_n(t) & \leq & \underbrace{(d+2)\|x\|^2+C_2 T+c_{B,p} e^T\,T \sup_{\overline{B}_{N_0}}|\rho|^{\frac{2}{p}}\cdot 
\big (\sum_{i=1}^{d}\|g_i\|^2_{L^{p}(B_{N_0})}\big )
+ 2dC_1 T^2}_{=:C_{T,B}}\\
&&+ (\underbrace{2dC_1T+C_2}_{=:C_T}) \int_{0}^{t} h_n(u) du.
\end{eqnarray*}
By Gronwall's inequality, $h_n(t) \leq C_{T,B}\cdot e^{C_{T} \cdot t}$. Since none of the involved constants depends on $n$, we can use Fatou's lemma letting $n \rightarrow \infty$, and obtain
\[
\E_{x}\left[\sup_{s \leq t} \|X_s \|^2\right] \leq C_{T,B} e^{C_{T} \cdot t}, \quad \forall t\le  T.
\]
Since $x \in \overline{B}$ was arbitrary, the desired result follows.
\end{proof}

\subsubsection{Non-explosion criteria involving the density $\rho$}\label{explform}
By \cite[Proposition 1.10]{St99}(a) we know that \eqref{conservative} holds, whenever
\begin{eqnarray}\label{consl1}
a_{ij},g_i-\beta^{\rho,A}_i\in L^1(\R^d,m), \quad 1\le i,j\le d.
\end{eqnarray}
Thus \eqref{consl1} provides a sufficient condition for non-explosion of $\M$ as in Theorem \ref{existhunt} by \eqref{prop1.5eq} which obviously depends on the knowledge of the density $\rho$.\\
A systematic study of non-explosion conditions, more precisely results implying (\ref{conservative}) and involving the density $\rho$ can be found in \cite[Corollary 15]{GT17}. 
\subsection{Recurrence criteria and other ergodic properties involving and not involving the density $\rho$}\label{subsecrec}
The measure $m=\rho\,dx$, where the density $\rho$ is as at the beginning of Section \ref{3} or as in Theorem \ref{1-3}, can be seen to define a stationary distribution. In fact, if the  $L^1 (\R^d, m)$-semigroup $(\widehat{T}_t)_{t>0}$ is conservative, for instance if 
\eqref{consl1} holds or there exists a constant $M\ge 0$ and some $N_0\in \N$, such that 
\begin{eqnarray*}
-\frac{\langle A(x)x, x \rangle}{ \left \| x \right \|^2 +1}+ \frac12\mathrm{trace}(A(x))+ \big \langle \left (2\beta^{\rho,A}-\mathbf{G}\right )(x), x \big \rangle \leq M( \left \| x \right \|^2 +1)( {\rm ln}(\left \| x \right \|^2+1)+1)
\end{eqnarray*}
for a.e. $x\in \R^d\setminus B_{N_0}$, as one can see from the dual version of Theorem \ref{nonextheo} or \cite[Proposition 1.10(c)]{St99}, then $m$ is an invariant measure (for $(T_t)_{t>0}$), i.e. for any $f\in L^1(\R^d,m)$
$$
\int_{\mathbb{R}^d} T_t f \,dm=\int_{\mathbb{R}^d} f \widehat{T}_t  1_{\mathbb{R}^d}\, dm=\int_{\mathbb{R}^d} f \,dm
$$
so that for any $A\in \mathcal{B}(\R^d)$ and $t\ge 0$
\begin{eqnarray*}
\P_{m}(X_t\in A)&:= &\int_{\mathbb{R}^d} \P_x(X_t\in A)\, m(dx)= \int_{\mathbb{R}^d} T_t 1_A (x)\, m(dx)\\
& =& \lim_{n\to \infty} \int_{\mathbb{R}^d} T_t 1_{A \cap B_n}(x)\, m(dx)= \lim_{n\to \infty} \int_{\mathbb{R}^d} 1_{A \cap B_n}(x)\, m(dx)= m(A).
\end{eqnarray*}
However, usually $m$ is not a probability measure, hence $\P_m$ is also not such a measure. But if it is, then $\P_m$ 
is a stationary distribution (if $(\widehat{T}_t)_{t>0}$ is conservative). Main parts of the  monograph \cite{BKRS} focus on the density $\rho$ or more generally on $m$, in case $m$ is a probability measure and aim in deriving properties of both (since both are in general not explicit). \\
We will first consider possibly infinite $m$ and we may assume that $\rho$ is explicit as is explained in the following remark.
\begin{rem}\label{rhoexpl}
All results up to now and further hold exactly in the same form, if we assume that $\rho \in C^{0,1-d/p}_{loc}(\R^d) \cap H^{1,p}_{loc}(\R^d)$ for some $p>d$ with $\rho(x)>0$ for all $x\in \R^d$ is explicitly given from the beginning, that $A:=(a_{ij})_{1\leq i,j \leq d}$ is as in  Theorem \ref{1-3} and that $\mathbf{B}=(b_1,\ldots , b_d)\in L^p_{loc}(\R^d,\R^d)$ satisfies
\begin{eqnarray}\label{41bgen}
\int_{\R^d} \langle \mathbf{B},\nabla f\rangle\,dm=0, \ \ \ \ \forall f \in C^{\infty}_0(\R^d).
\end{eqnarray}
Indeed, we then just have to set $\mathbf{G}:=\beta^{\rho,A}+\mathbf{B}$. Then all conclusions of Theorem \ref{1-3} hold with the explicitly chosen density from above. Note that this also includes the setting of Theorem \ref{1-3} since by its conclusion a $\rho$ like above exists and can hence be \lq\lq explicitly\rq\rq\ chosen. 
\end{rem}
We want to derive explicit conditions for recurrence involving and not involving the density $\rho$ in two general cases where $m$ is a general $\sigma$-finite measure and where $m$ is a finite, yet without loss of generality a probability measure. First, we derive a lemma which leads to irreducibility in the probabilistic sense and strict irreducibility in the sense of generalized Dirichlet forms (see Corollary \ref{irreduci}) and 
as a byproduct leads to a weaker condition for non-explosion (see Remark \ref{conspoint}).
\begin{lem}\label{irreduci0}
Under the assumptions of Remark \ref{rhoexpl}, it holds: 
\begin{itemize}
\item [(i)] Let $A \in \mathcal{B}(\R^d)$ be such that $P_{t_0} 1_A (x_0)=0$ for some $t_0>0$ and $x_0\in \R^d$. Then $m(A)=0$.
\item [(ii)] Let $A \in \mathcal{B}(\R^d)$ be such that $P_{t_0} 1_A (x_0)=1$ for some $t_0>0$ and $x_0\in \R^d$.  Then $P_t 1_{A}(x)=1$ \;for all $(x,t) \in \R^d \times (0,\infty)$.
\end{itemize}
\end{lem}
\begin{proof}
(i) Suppose $m(A)>0$. Choose an open ball $B_{r}(x_0)\subset \R^d$ such that
$$
0<m\left(A\cap B_{r}(x_0)\right)< \infty. 
$$
Let $u:=\rho P_{\cdot} 1_{A \cap B_{r}(x_0)}$. Then $0=u(x_0, t_0) \leq \rho(x_0) P_{t_0} 1_{A}(x_0)=0$. Take $f_n \in C_0^{\infty}(\R^d)$ with $f_n \geq 0$ such that $f_n \rightarrow 1_{A \cap B_{r}(x_0)}$ in $L^1(\R^d, m)$. Then by \eqref{rest} and the explanation right after it, for arbitrary bounded open set $U$ in $\R^d$ and $[\tau_1, \tau_2] \subset (0, \infty)$, there is some $\gamma \in (0,1)$ such that
\[
P_{\cdot} f_n (\cdot) \rightarrow P_{\cdot}1_{A \cap B_{r}(x_0)}(\cdot) \; \text{ in } C^{\gamma; \frac{\gamma}{2}}(\overline{U} \times [\tau_1, \tau_2]),
\]
hence
\begin{equation} \label{app}
u_n:= \rho P_{\cdot}f_n \rightarrow u \; \text{ in } C^{\gamma; \frac{\gamma}{2}}(\overline{U} \times [\tau_1, \tau_2]).
\end{equation}
Fix $T>t_0$ and $U\supset \overline{B}_r(x_0)$. Then (see proof of Theorem \ref{1-3reg}) for all $\varphi \in C_0^{\infty}(U\times (0, T))$
\begin{eqnarray}\label{sol2}
\int_0^T\int_{U} \left ( \frac{1}{2}\langle A \nabla u_n, \nabla \varphi \rangle + u_n \langle \beta, \nabla \varphi \rangle -u_n\partial_t\varphi \right ) dxdt=0,
\end{eqnarray}
where $\beta$ is defined as in the proof of Theorem \ref{1-3reg}. Now take arbitrary but fixed $(x,t) \in B_{r}(x_0) \times (0, t_0)$.
By \cite[Theorem 5]{ArSe}
\[
0 \leq u_n(x,t) \leq u_n(x_0, t_0)\;\exp \left ( C \Big(\frac{\|x_0-x\|^2}{t_0-t}+ \frac{t_0-t}{\min(1,t)} +1 \Big)\right)
\]
and \eqref{app} applied with $U\supset \overline{B}_r(x_0)$, $[\tau_1,\tau_2]\supset  [t,t_0]$ then leads to
\[
0 \leq u(x,t) \leq u(x_0, t_0)\;\exp \left ( C \Big(\frac{\|x_0-x\|^2}{t_0-t}+ \frac{t_0-t}{\min(1,t)} +1 \Big)\right)=0.
\]
Thus, $ P_t 1_{A\cap B_{r}(x_0)}(x)=0$ for all $x \in B_{r}(x_0)$ and $0<t<t_0$, so that
\[
0=\int_{\R^d} 1_{A \cap B_{r}(x_0)} P_t 1_{A\cap B_{r}(x_0)} dm \underset{\text{as } t \rightarrow 0}{ \; \longrightarrow} m(B_r(x_0) \cap A)>0,
\]
which is contradiction. Therefore, we must have $m(A)=0$.\\
(ii) Let $y \in \R^d$ and $0<s<t_0$ be arbitrary but fixed and let $r:=2\|x_0-y\|$ and let $B$ be any open ball. Take $g_n \in C_0^{\infty}(\R^d)$ with $0 \leq g_n \leq 1$ such that $g_n \rightarrow 1_{A \cap B}$ in $L^1(\R^d, m)$. Then by \eqref{rest} and the explanation right after it, there is some $\gamma \in (0,1)$ such that
\begin{equation} \label{appro}
P_{\cdot} g_n (\cdot) \longrightarrow P_{\cdot} 1_{A\cap B} (\cdot) \; \text{ in } C^{\gamma; \frac{\gamma}{2}}(\overline{B}_{r}(x_0) \times [s/2, 2t_0]).
\end{equation}
Fix $T>0$ and $U \supset \overline{B}_{r}(x_0)$. Using the property
$$
\beta=\frac{1}{2}\nabla A+\mathbf{G}-2\beta^{\rho, A}= \mathbf{B}- \beta^{\rho,A} +\frac{1}{2} \nabla A=\mathbf{B}-\frac{1}{2}A \frac{\nabla \rho}{\rho},
$$
and (\ref{41}), we directly get  for all $\varphi \in C_0^{\infty}(U \times (0, T))$
\begin{align} \label{1zero}
\int_0^T\int_{U} \left ( \frac{1}{2}\langle A \nabla \rho, \nabla \varphi \rangle + \rho \langle \beta, \nabla \varphi \rangle -\rho\partial_t\varphi \right ) dxdt =\int_0^T \left(\int_{U} \langle \mathbf{B}, \nabla \varphi \rangle \rho dx \right) dt =0,
\end{align}
and  (cf. the proof of Theorem \ref{1-3reg})  we also get
\begin{align} \label{2zero}
\int_0^T\int_{U} \left ( \frac{1}{2}\langle A \nabla \left(\rho P_{\cdot} g_n\right), \nabla \varphi \rangle + (\rho P_{\cdot} g_n) \langle \beta, \nabla \varphi \rangle -(\rho P_{\cdot} g_n) \partial_t\varphi \right ) dxdt=0.
\end{align}
Now let $u_n(x,t):=\rho(x) \left(1- P_t g_n (x) \right)$. Then $u_n \in H^{1,2}(U \times (0,T))$ and $u_n \geq 0$. Subtracting $\eqref{2zero}$ from $\eqref{1zero}$  implies
$$
\int_0^T\int_{U} \left ( \frac{1}{2}\langle A \nabla u_n, \nabla \varphi \rangle + u_n \langle \beta, \nabla \varphi \rangle -u_n\partial_t\varphi \right ) dxdt=0.
$$
Thus, by \cite[Theorem 5]{ArSe}
\[
0 \leq u_n(y,s) \leq u_n(x_0, t_0)\; \underbrace{ \exp \left ( C \Big(\frac{\|x_0-y\|^2}{t_0-s}+ \frac{t_0-s}{\min(1,s)} +1 \Big) \right)}_{=:C_2}.
\]
By \eqref{appro}
\[
0 \leq\rho(y) \left(1- P_s 1_{A \cap B} (y) \right) \leq C_2 \rho(x_0) \left(1- P_{t_0} 1_{A \cap B} (x_0) \right).
\]
Note that for all $(x,t) \in \R^d \times(0,\infty)$, $P_t 1_{A \cap B_n} 1(x) \nearrow P_t 1_{A} (x)$ as $n \rightarrow \infty$.
Thus,
\[
0 \leq\rho(y) \left(1- P_s 1_{A} (y) \right) \leq C_2 \rho(x_0) \left(1- P_{t_0} 1_{A} (x_0) \right)=0.
\]
Consequently, $P_{s}1_{A}(y)=1$ for any $(y, s) \in \R^d \times (0, t_0)$ which can be extended to  $\R^d \times (0, t_0]$ by continuity. And by the sub-Markovian property, $P_{t_0}1_{\R^d}(y)=1$ for any $y \in \R^d$. Now let $t \in (0, \infty)$ be given. Then there exists $k \in \N \cup \left \{0 \right \}$ such that
\[
k t_0<t \leq  (k+1) t_0
\]
and so $P_t 1_{A} = P_{kt_0+ (t-kt_0)} 1_{A}= \underbrace{P_{t_0} \circ \cdots \circ P_{t_0}}_{k\text{-times}} \circ P_{t-kt_0} 1_{A} =1$.
\end{proof}
\begin{rem}\label{conspoint}
Under the assumptions of Remark \ref{rhoexpl}, 
by Lemma \ref{irreduci0}(ii) we know that  $\M$ is non-explosive, if $\P_x(\zeta=\infty)=1$ for some $x\in \R^d$.
More precisely, if $\P_{x_0}(X_{t_0} \in \R^d)=1$ for some $(x_0, t_0) \in \R^d \times (0, \infty)$, then $\M$ is non-explosive. This (together with Proposition \ref{regular2}, Lemma \ref{fini1}) generalizes and improves  \cite[Lemma 2.5]{Bha} to possibly locally unbounded drift coefficient using a completely different and genuine proof.
\end{rem}
\medskip
$A\in \mathcal{B}(\R^d)$ is called {\it weakly invariant} relative to $(T_t)_{t>0}$, if
$$
T_t (f \cdot 1_A ) (x)=0,\ \  \text{for } m\text{-a.e.}  \ \ x\in \R^d\setminus A,
$$
for any $t> 0$, $f\in L^2(\R^d,m)$. $(T_t)_{t>0}$ is said to be {\it strictly irreducible}, if for any weakly invariant set $A$ relative to $(T_t)_{t>0}$, we have $m(A)=0$ or $m(\R^d\setminus A)=0$.\\
\begin{cor}\label{irreduci}
Under the assumptions of Remark \ref{rhoexpl}, it holds:
\begin{itemize}
\item[(i)] $(T_t)_{t>0}$ is strictly irreducible.  
\item[(ii)] Let $A \in \mathcal{B}(\R^d)$ with $m(A)>0$. Then $\P_x(X_t\in A)>0$ for all $x\in \R^d, t>0$, i.e. $\M$ is irreducible in the probabilistic sense.
\end{itemize}
\end{cor}
\begin{proof}
(i) Let $A \in \mathcal{B}(\R^d)$ be a weakly invariant set with $m(\R^d\setminus A)\not=0$. Then by monotone  approximation with the $L^2$-functions $1_{B_n}$, $n\ge 1$, we get for any $t>0$
$P_t  1_A  (x)=0$, for  $m$-a.e.  $x\in \R^d\setminus A$. Then there exists $t_0>0$ and $x_0 \in \R^d \setminus A$ such that $P_{t_0} 1_A(x_0) =0$. By Lemma \ref{irreduci0}(i), we have $m(A)=0$, as desired.\\
(ii) By contraposition of Lemma \ref{irreduci0}(i), $\mathbb{P}_x\left(X_t \in A \right) = P_t 1_A(x) >0$, for all $x \in \R^d, t>0$.
\end{proof}
\subsubsection{Explicit recurrence criteria for possibly infinite $m$}\label{recurrencegeneral}
We continue with some further definitions. Define the last exit time $L_A$ from $A\in \mathcal{B}(\R^d)$ by
$$
L_A:=\sup \{ t\geq 0: X_t\in A\},\ \ (\sup \emptyset:=0).
$$
$\M$ is called {\it recurrent (in the probabilistic sense)}, if for any $\emptyset\not=U\subset \R^d$, $U$ open, we have 
\begin{eqnarray}\label{classicalrecurrence}
\P_x(L_U=\infty)=1,\ \ \forall x\in \R^d.
\end{eqnarray}
Let $(\vartheta_t)_{t\ge 0}$ be the shift operator of $\M$. Using the shift invariance of $\Lambda:=\{L_U=\infty\}$, the Markov property and the strong Feller property of $(P_t)_{t>0}$, we get for all $x\in \R^d$, $t>0$
\[\mathbb{P}_x(\Lambda)  =  \mathbb{P}_x(\vartheta_t^{-1}(\Lambda))= \mathbb{E}_x[\mathbb{E}_x[ 1_{\Lambda}\circ \vartheta_t \,|\, \F_t]]
=  \mathbb{E}_x[\mathbb{E}_{X_t}[ 1_{\Lambda}]]=P_t \mathbb{E}_{\cdot}[1_{\Lambda}](x).
\]
Thus
\begin{eqnarray}\label{GimTr recurrence}
\eqref{classicalrecurrence}\ \Longleftrightarrow \ \P_x(L_U=\infty)=1\quad \text{for }m\text{-a.e. } x\in \R^d.
\end{eqnarray}
The following is now an easy consequence of the results obtained here, in \cite{GT2} and \cite{Ge}.
\begin{prop}\label{rectrans} 
Consider the assumptions of Remark \ref{rhoexpl}. Then $(T_t)_{t>0}$ (or equivalently $\M$) is either transient or recurrent in the sense of \cite{GT2}.
\begin{itemize}
\item[(i)] Suppose $(T_t)_{t>0}$ is transient  in the sense of \cite{GT2}. Then for any compact $K\subset \R^d$, it holds $\P_{x}(L_{K} < \infty)=1$ for all  $x \in \R^d$. In particular
\begin{eqnarray}\label{wanderinginfty}
\P_x(\lim_{t\to \infty} X_t=\Delta \text{ in } \R^d_{\Delta})=1 \ \text{for any } x \in \R^d.
\end{eqnarray}
\item[(ii)] Suppose $(T_t)_{t>0}$ is recurrent in the sense of \cite{GT2}. Then $\M$ is non-explosive and recurrent (in the probabilistic sense), i.e. \eqref{classicalrecurrence} holds for any nonempty open  $U\subset \R^d$.
\end{itemize}
\end{prop}
\begin{proof} The first assertion follows from Corollary \ref{irreduci}(i) and \cite[Remark 3(b)]{GT2}. \\ 
(i) Applying \cite[Lemma 6]{GT2} and the last part of Lemma \ref{fini1}(i) we get the existence of $g \in L^1(\R^d, m)\cap L^\infty(\R^d, m)$ with $g>0$ everywhere, such that $Rg:=\E_{\cdot} \left[\int_{0}^{\infty} g(X_t)   dt \right]\in L^\infty(\R^d, m)$. Using that $Rg$ is lower semicontinuous by the strong Feller property and essentially bounded, we deduce $Rg(x)< \infty$ for any $x\in \R^d$. Obviously, $0<Rg(x)$ for any $x\in \R^d$. Modifying the proof of \cite[Proposition 10]{GT2} (which originates from \cite{Ge}) with the open sets $U_n:=\{Rg>\frac{1}{n}\}$, $n\ge 1$, and using the strong Feller property of $(P_t)_{t>0}$, we obtain $\P_{x}(L_{U_n} < \infty)=1$ for all  $x \in \R^d, n\ge 1$. Now the first assertion follows easily since $(U_n)_{n\ge 1}$ is an open cover of any compact set $K\subset \R^d$. The second assertion follows from the first since the paths of $\M$ are continuous on the one point compactification $\R^d_{\Delta}$.\\
(ii) (\ref{conservative}) is a consequence of \cite[Corollary 20]{GT2} and $\M$ is hence non-explosive by \eqref{prop1.5eq}. Moreover, the right hand side of \eqref{GimTr recurrence} holds for any $\emptyset\not=U\subset \R^d$, $U$ open, by \cite[Proposition 11(d)]{GT2}. Therefore $\M$ is recurrent in the probabilistic sense.
\end{proof}
\begin{rem}\label{rectransrem}
In Proposition \ref{rectrans}, we get actually equivalences in (i) and (ii). Namely, \eqref{wanderinginfty} implies that \cite[Condition (8) of Proposition 10]{GT2}  is satisfied. Thus \eqref{wanderinginfty} implies transience of $\M$ (or equivalently $(T_t)_{t>0}$) in the sense of \cite{GT2} by \cite[Proposition 10]{GT2}. Likewise, if $\M$ is recurrent (in the probabilistic sense), then it cannot satisfy \eqref{wanderinginfty}. Therefore, by Proposition \ref{rectrans}(i) and its first part, $(T_t)_{t>0}$ must be recurrent in the sense of \cite{GT2}.
\end{rem}
Define for $r\ge 0$,
\begin{equation*}
v_1(r):= \int_{B_r} \frac{\langle A(x)x, x \rangle}{\|x\|^2} m(dx), \ \ \ v_2(r):=\int_{B_r} \left |\langle \mathbf{B}(x),x\rangle \right | m(dx),
\end{equation*}
where $\mathbf{B}$ is defined as in Theorem \ref{1-3} and let 
$$
v(r):=v_1(r)+v_2(r), \ \ a_n:=\int_1^n \frac{r}{v(r)}dr, \ \ n\ge 1.
$$
\begin{theo}\label{rhorecurrence}(Corollary of \cite[Theorem 21]{GT2})
Consider the assumptions of Remark \ref{rhoexpl} and suppose  that
$$
\lim_{n\rightarrow \infty}a_n=\infty \ \ \ \text{and} \ \ \  \lim_{n\rightarrow \infty} \frac{\log(v_2(n)\vee 1)}{a_n}=0.
$$
Then $\M$ is recurrent (in the probabilistic sense), non-explosive and $m$ is an invariant measure for $(T_t)_{t\ge 0}$. 
\end{theo}
\begin{proof}
By \cite[Theorem 21]{GT2}  applied with $\rho(x)=\|x\|$ (the $\rho$ of \cite{GT2} is different from the $\rho$ defined here), the given assumption implies that $(T_t)_{t>0}$ is not transient in the sense of \cite{GT2}. Then apply Proposition \ref{rectrans}. The same considerations apply to $(\widehat{T}_t)_{t>0}$, so that $(\widehat{T}_t)_{t>0}$ is in particular conservative and then $m$ is an invariant 
measure for $(T_t)_{t\ge 0}$ (and also $(\widehat{T}_t)_{t>0})$. 
\end{proof}
\begin{lem}\label{exittimeofball}
Consider the assumptions of Remark \ref{rhoexpl}. Then for any $x\in \R^d$ and $N\in \N$, we have $\P_x(\sigma_{N}<\infty)=1$.
\end{lem}
\begin{proof}
Suppose to the contrary that there exists $N\in \N$ and $x\in \overline{B}_N$ such that $\P_x(\sigma_{N}=\infty)\ge \delta>0$. 
Then $\M$ is not recurrent in the probabilistic sense. Applying Proposition \ref{rectrans}, we obtain  
$
\P_{x}(L_{K} < \infty)=1 \; \text{ for all } x \in \R^d \text{ and any compact }K\subset \R^d.
$
Therefore  $\P_x(\sigma_{N}=\infty)\ge \delta>0$ cannot hold and the assertion follows.
\end{proof}\\
\centerline{}
The following theorem extends \cite[Chapter 6, Theorem 1.2]{Pi} to locally unbounded drift coefficient.
\begin{theo}\label{recurrencepinsky}
Consider the assumptions of Remark \ref{rhoexpl} and suppose that there exists a positive $\psi \in C^2(\R^d)$ and some $N_0\in \N$ such that $L\psi\le 0$ a.e. on $\R^d\setminus B_{N_0}$ 
and $\inf_{\partial B_n}\psi \to \infty$ as $n\to \infty$. Then $\M$  of Theorem \ref{existhunt} is recurrent (in the probabilistic sense) and non-explosive. In particular, the assumptions above are satisfied (take $\psi(x)=\ln\left (\|x\|^2+1\right )+1$), if there is some $N_0\in \N$, such that 
\begin{eqnarray}\label{explicitrecurrencegeneral}
-\frac{\langle A(x)x, x \rangle}{ \left \| x \right \|^2+1}+ \frac12\mathrm{trace}A(x)+ \big \langle \mathbf{G}(x), x \big \rangle \leq 0
\end{eqnarray}
for a.e. $x\in \R^d\setminus B_{N_0}$. 
\end{theo}
\begin{proof}
Clearly, $\M$ is non-explosive by Remark \ref{consstannat}(iii). Let $n\ge N_0$ and $x\in \R^d\setminus \overline{B}_n$ be arbitrary. Choose any $N\in \N$ with $x\in B_N$. 
We will first show that $\P_x(\sigma_{B_{n}}<\infty)=1$. Using that $L\psi\le 0$ a.e. on $\R^d\setminus B_{N_0}$ we can see that
$$
\E_x[\psi(X_{t\wedge \sigma_{B_{n}}\wedge\sigma_N})]\le \psi(x).
$$
Since $\P_x(\sigma_{N}<\infty)=1$ by Lemma \ref{exittimeofball}, we can let $t\to \infty$ and obtain with elementary calculations (cf. for instance the proof of Theorem \ref{nonextheo})
\begin{eqnarray*}
(\inf_{\partial B_N}\psi)\cdot \P_x(\sigma_{B_{n}}=\infty)\le \E_x[\psi(X_{\sigma_N})1_{\{\sigma_{B_{n}}=\infty\}}]\le \E_x[\psi(X_{\sigma_{B_{n}}\wedge\sigma_N})]\le \psi(x).
\end{eqnarray*}
Letting $N\to \infty$ and using the further assumption on $\psi$, we get $\P_x(\sigma_{B_{n}}=\infty)=0$ and the claim is shown. From now on let $n:=N_0+1$. Then obviously $\P_x(\sigma_{B_{n}}<\infty)=1$ for any $x\in B_n$ and by the claim $\P_x(\sigma_{B_{n}}<\infty)=1$ for any $\R^d\setminus \overline{B}_n$. If $x\in \partial B_n$, then by the claim again $\P_x(\sigma_{B_{N_0}}<\infty)=1$ and since $\sigma_{B_{N_0+1}}\le \sigma_{B_{N_0}}$, we finally get
$$
\P_x(\sigma_{B_{n}}<\infty)=1\quad \text{for any } x\in \R^d.
$$
Let $z\in \R^d, s>0$ be arbitrary. Then by the Markov property and since $\M$ is non-explosive
\begin{eqnarray}\label{recuno}
\P_z(X_t\in B_{n} \text{ for some } t\in[s,\infty))=\P_z(\sigma_{B_{n}}\circ \vartheta_s<\infty)=\E_z[\P_{X_s}(\sigma_{B_{n}}<\infty)]=1.
\end{eqnarray}
Hence $\P_z(L_{\overline{B}_{N_0+1}}<\infty)=0$ and the assertion now follows from Proposition \ref{rectrans}. 
\end{proof}
\subsubsection{Uniqueness of invariant probability measures and ergodic properties in case $m$ is a probability measure}\label{uniqueinv}
In this subsection, we suppose (except at the very end of it) that $m$ is  a finite measure. Dividing by a normalizing constant, which will not change the generator $L$, we may without loss of generality assume that $m$ is a probability measure. Coming back to the situation at the beginning of Section \ref{subsecrec}, we have the following:
\begin{rem}\label{invarianceequivalent}
 If $m$ is a probability measure, then $m$ is $(T_t)_{t>0}$-invariant, if and only if it is $(\widehat{T}_t)_{t>0}$-invariant  (cf. \cite[Proposition 1.10(b)]{St99}). In either case $\P_m$ is then a stationary distribution.
\end{rem}
It is clear that the $(\widehat{T}_t)_{t>0}$-invariance of $m$ is equivalent to the conservativeness of $(T_t)_{t>0}$, i.e. to \eqref{conservative}. Consequently, using Remark \ref{invarianceequivalent}, we see that $m$ is an invariant (probability) measure for $(T_t)_{t>0}$, if \eqref{conservative} holds. Therefore, \eqref{conservative2} provides an explicit criterion for 
$m$ to be an invariant (probability) measure. Now, we have the following:
\begin{theo}\label{ergodicinvariant}
 Let $\M$ be as in Theorem \ref{existhunt}. Suppose that $m$ is a probability measure and  that \eqref{conservative} holds. Then:
\begin{itemize}
\item[(i)] $m$ is strongly mixing (cf. \cite{DPZB}) and for arbitrary $x\in \R^d$ and $A \in \mathcal{B}(\R^d)$
$$
\lim_{t\to \infty}\P_x(X_t\in A)=m(A).
$$
\item[(ii)] $m$ is the unique probability measure that is $(P_t)_{t>0}$-invariant (in the sense of \cite{DPZB}). 
\item[(iii)] $m$ is equivalent to $\P_x\circ X_t^{-1}$ for any $(x,t)\in \R^d\times (0,\infty)$.
\item[(iv)] Let $A \in \mathcal{B}(\R^d)$ be such that $m(A)>0$ and $(t_n)_{n\ge 1}\subset (0,\infty)$ be any sequence with $\lim_{n\to \infty}t_n=\infty$. Then
$$
\P_x(X_{t_n}\in A \text{ for infinitely many } n\in \N)=1,\ \ \forall x\in \R^d.
$$
In particular, $\M$ is recurrent.
\end{itemize}
\end{theo}
\begin{proof}
By Theorem \ref{1-3reg}, Lemma \ref{fini1}(i) and Corollary \ref{irreduci}(ii), $(P_t)_{t>0}$  is strong Feller and $\M$ is irreducible (in the probabilistic sense). Then \cite[Proposition 4.1.1]{DPZB} implies that $(P_t)_{t>0}$ is regular. Therefore the assertions (i)-(iii) follow by Doob's Theorem, see  \cite[Theorem 4.2.1]{DPZB}. Then using (i),  assertion (iv) follows by \cite[Theorem 3.4.5]{DPZB}.
\end{proof}
\begin{rem}\label{explicitergodicity}
Assume that as in Remark \ref{rhoexpl}, $\rho$, $A$, $\mathbf{B}$ are explicitly given and that $m=\rho\,dx$ is a probability measure such that \eqref{conservative} holds. Then Theorem \ref{ergodicinvariant} applies. 
This result seems to be new even if $\mathbf{B}\equiv 0$.
\end{rem}
For the rest of the section we do {\it not assume that $m$ is a finite measure} and present a condition that is independent of $\rho$ and makes Theorem \ref{ergodicinvariant} applicable. The following proposition is a variant of \cite[Chapter 6, Theorem 1.3]{Pi} which can be applied to locally unbounded drift coefficients.
\begin{prop}\label{indeprho} Under the assumptions of Theorem \ref{1-3}, suppose that there exists a positive $\psi \in C^2(\R^d)$, some $N_0\in \N$ and $C>0$, such that $L\psi\le -C$ a.e. on $\R^d\setminus B_{N_0}$ 
and $\inf_{\partial B_n}\psi \to \infty$ as $n\to \infty$. 
Then $m$ is finite and $\M$ is non-explosive. In particular, 
\eqref{conservative} holds and by normalizing $m$ if necessary, we can see that the assumptions of Theorem \ref{ergodicinvariant} are satisfied. Thus Theorem \ref{ergodicinvariant}(i)-(iv) hold. In particular, the assumptions above are satisfied (take $\psi(x)=\ln\left (\|x\|^2+1\right )+1$), if there exists a constant  $C> 0$ and some $N_0\in \N$, such that 
\begin{eqnarray}\label{conservative9}
-\frac{\langle A(x)x, x \rangle}{ \left \| x \right \|^2 +1}+ \frac12\mathrm{trace}A(x)+ \big \langle \mathbf{G}(x), x \big \rangle \leq -C\left ( \left \| x \right \|^2+1\right )
\end{eqnarray}
for a.e. $x\in \R^d\setminus B_{N_0}$. 
\end{prop}
\begin{proof}
Using $L\psi(x) \leq -C$ for a.e. $x\in \R^d\setminus B_{N_0}$, the finiteness of $m$ follows by \cite[Corollary 2.3.3]{BKRS} or \cite[Theorem 2]{BRS} for the original result. Since $L\psi(x) \leq M\psi(x)$ for a.e. $x\in \R^d\setminus B_{N_0}$ for any $M>0$, $\M$ is non-explosive by Remark \ref{consstannat}(iii). We may hence assume that the conditions of Theorem \ref{ergodicinvariant} are satisfied.
\end{proof}\\
\centerline{}
In the next example, we shall give a sufficient condition for  \eqref{conservative9} to hold.
\begin{exam}\label{exam2}
Let $I$ be the identity matrix consisting of ones on the diagonal and zeros outside and set $A(x):=\Psi(x) I$ where $\Psi(x) \in H^{1,p}_{loc}(\R^d) \cap C_{loc}^{1-d/p}(\R^d)$ with $\Psi(x)>0$ for all $x \in \R^d$. Let $\phi_1 \in L_{loc}^p(\R^d)$, $\phi_1 \geq 0$ a.e. and $\mathbf{G}(x):= \left( -\phi_1(x)1_{\R^d \setminus B_{N_0}} +\phi_2(x) 1_{B_{N_0}} \right)x$ for some $\phi_2 \in L^p_{loc}(\R^d)$. Suppose that for some $N_0\in \N\cup\{0\}$,
\begin{equation} \label{condition1}
\frac{d}{2} \Psi(x) +C(\|x\|^2+1) \leq \phi_1(x) \|x\|^2  \text{ a.e. $x\in \R^d\setminus B_{N_0}$ }.
\end{equation}
Then \eqref{condition1} implies \eqref{conservative9}. 
\end{exam}
Now we compare our results with results of \cite{ZhXi16}.
\begin{rem}\label{compZhZi}
As one can see from the proof of Theorem \ref{ergodicinvariant} in order to derive the conclusions Theorem \ref{ergodicinvariant}(i)-(iv) one needs for instance  the classical strong Feller property  of $(P_t)_{t>0}$ and the irreducibility of $\M$. In our case, these are directly implied under the conditions of Theorem \ref{1-3} (cf.  Theorem \ref{1-3reg} and  Corollary \ref{irreduci}(ii)). But the conditions to obtain the strong Feller property and irreducibility in \cite{ZhXi16} are quite strong, and there are many cases where  \eqref{conservative9} is satisfied but one cannot obtain the strong Feller property nor irreducibility from the results of \cite{ZhXi16}. The following provides a comparison of \eqref{conservative9} and the rather strong conditions of \cite{ZhXi16}:
\begin{itemize}
\item[(i)] \begin{itemize} \item[a)] If $\mathbf{G}$ is not bounded on an open ball, in order to get the strong Feller property and the irreducibility, \cite[Theorem 1.7]{ZhXi16} needs very strong conditions \cite[(H1'), (H2')]{ZhXi16} such as global uniform ellipticity and boundedness of $A$ and Lipschitz continuity of $A, \mathbf{G}$ and the growth condition $\|\mathbf{G}(x)\| \leq C(1+\|x\|)$  outside an open ball.  For example if we take $A(x)=(1+\|x\|)I$ and $\phi_1(x)=\|x\|^2$, then \eqref{condition1} holds, but (H1') and (H2') in \cite{ZhXi16} are both not satisfied. Thus the conditions of  \cite{ZhXi16} do neither provide global well-posedness, nor strong Feller properties, nor irreducibility and so on, whereas we get the full conclusions of Proposition \ref{indeprho}.
\item[b)] If $\mathbf{G}$ is locally bounded on $\R^d$, to get the strong Feller property and the irreducibility,  \cite[Theorem 1.2]{ZhXi16} also requires quite strong conditions. For example, a diffusion matrix with strong decay such as $A(x)=\exp(-\exp(\|x\|^2)) I$ cannot be handled by results of \cite{ZhXi16}, since \cite[(1.4)]{ZhXi16} is not satisfied, but we do not have such restrictions. Moreover, if $A(x)=I$ and $\phi_1(x)=\exp(\exp(\|x\|^2)))$, then clearly \eqref{condition1} is satisfied, but  \cite[(1.7)]{ZhXi16} is not satisfied. Note that \cite[(1.6), (1.8)]{ZhXi16} requires $A$ to be (besides an $H^{1,q}_{loc}$-condition, $q>d+2$) locally Lipschitz outside an open ball, if $b\equiv 0$ in \cite{ZhXi16}),  which is also stronger than our condition $a_{ij} \in H_{loc}^{1,p}(\R^d)$ for $1 \leq i,j \leq d$ for some $p>d$.
\end{itemize}
\item[(ii)]
We will give a simple example which has a global pathwise unique solution satisfying all conclusions of Proposition \ref{indeprho}, but the non-explosion conditions in  \cite{ZhXi16} do even not allow to obtain the existence of global solution. 
Choose $\Psi(x)=\phi_1(x)=(1+\|x\|)^2$. Then \eqref{condition1} is satisfied, so that by Example \ref{exam2} we may apply Proposition \ref{indeprho} and get a global pathwise unique solution satisfying (i)-(iv) of Theorem \ref{ergodicinvariant}. Now consider 
$$
\phi_2(x)=\frac{1}{\|x-(\frac{N_0}{2},0, \ldots,0)\|^{d/(p+1)}}, \quad x\in \R^d.
$$
Then $\phi_2 \in L^p_{loc}(\R^d)$ and $\lim_{x\to (\frac{N_0}{2},0, \ldots,0)} \phi_2(x) = \infty$, so that $\mathbf{G}$ as defined in Example \ref{exam2} satisfies
\[
\big \langle \mathbf{G}(x), x \big \rangle \longrightarrow \infty \text{ as } x\to (\frac{N_0}{2},0, \ldots,0). 
\]
Thus, the non-explosion condition \cite[(1.5)]{ZhXi16} is not satisfied and obviously global boundedness of $A$ and linear growth of $\|\mathbf{G}\|$ do not hold, which means  \cite[\text{[H1'] [H2']}]{ZhXi16} are not satisfied. In particular, no non-explosion condition of  \cite{ZhXi16} holds.
\item[(iii)] By our method we have directly a candidate for invariant measure, namely $m$. In \cite{ZhXi16} no candidate for invariant measure can be deduced.
\end{itemize}
\end{rem}

\section{An application to pathwise uniqueness and strong solutions: the main theorem}\label{5}
We present here an application of our results to pathwise uniqueness and existence of strong solutions up to $\infty$. 
\begin{theo}\label{unique2}
Let $A=(a_{ij})_{1\leq i,j \leq d}$, $\mathbf{G}$, be as in Theorem \ref{1-3}, such that $A=\sigma\sigma^T$ and $\sigma=(\sigma_{ij})_{1 \le i,j \le d}$ satisfies $\sigma_{ij}\in H^{1,p}_{loc}(\R^d)$, $1\leq i,j \leq d$. Suppose that (\ref{conservative2}) holds for $A$ and $\mathbf{G}$ (or more generally any condition that guarantees the non-explosion of $\M$ of Theorem \ref{existhunt}).  Then the stochastic differential equation
\begin{eqnarray}\label{unique344}
X_t= x_0 + \int_0^t\sigma(X_s)dW_s + \int^{t}_{0} \mathbf{G}(X_s) ds, \ 0\le t<\infty,\ x_0\in \mathbb{R}^d, 
\end{eqnarray}
where $W = (W^1,\dots,W^d)$  is a standard  $d$-dimensional Brownian motion starting from zero, has a pathwise unique and strong solution $(X_t)_{t\ge 0}$. In particular, and without any further assumption, $(X_t)_{t\ge 0}$ satisfies more than classical strong Feller properties (see Theorem \ref{1-3reg}, Proposition \ref{regular2} and Lemma \ref{fini1}(i)), has integrability properties as in Lemma \ref{nest}, is irreducible (by Corollary \ref{irreduci}), satisfies the long time behavior as in Proposition \ref{rectrans} and Remark \ref{rectransrem}, and has further additional properties like in Lemma \ref{irreduci0}, Remark \ref{conspoint}, Lemma \ref{exittimeofball}. Moreover, there are diverse explicit further conditions to guarantee moment inequalities, recurrence and ergodicity, including existence and uniqueness of invariant probability measures for $(X_t)_{t\ge 0}$, see Theorems \ref{supestimate}, \ref{rhorecurrence}, \ref{recurrencepinsky} and Proposition \ref{indeprho} (or for the more general statement see Theorem \ref{ergodicinvariant}). Finally, Lemma \ref{fini1}(ii) also holds for $(X_t)_{t\ge 0}$.
\end{theo}
\begin{proof}
The existence of a weak solution up to $\zeta=\infty$ under the present assumptions follows from Theorems \ref{weakexistence}(i) and \ref{nonextheo}. This then leads to the existence of a pathwise unique and strong global solution by \cite[Theorem 1.3]{Zh11} and \cite[Chapter IV.1. Theorem 1.1]{IW89}. For all other statements, we refer to the corresponding proofs and note that uniqueness in law is implied by pathwise uniqueness (see for instance \cite[Chapter IV.1. Corollary]{IW89}), so any solution to \eqref{unique344} has the same law.
\end{proof}

\centerline{}
Haesung Lee, Gerald Trutnau\\
Department of Mathematical Sciences and \\
Research Institute of Mathematics of Seoul National University,\\
1 Gwanak-Ro, Gwanak-Gu,
Seoul 08826, South Korea,  \\
E-mail: fthslt@snu.ac.kr, trutnau@snu.ac.kr
\end{document}